\documentclass[11pt]{article}

\usepackage{amsthm,amsmath,amssymb}
\usepackage{amsthm,amsmath}
\usepackage[matrix,arrow]{xy}
\usepackage{graphicx}
\usepackage{enumerate,amssymb,setspace}
\usepackage{amsfonts,amsthm,amscd}
\usepackage{hyperref}
\date{}
 \usepackage{amscd,amssymb}
 \usepackage{amsthm,amsmath,enumerate,amssymb}
 \usepackage{longtable}
 \usepackage{comment}

\newtheorem{theorem}{Theorem}[section]
\newtheorem{lemma}[theorem]{Lemma}
\newtheorem{prop}[theorem]{Proposition}

\newtheorem{corollary}[theorem]{Corollary}
\newtheorem{cor}[theorem]{Corollary}
\newtheorem{conjecture}[theorem]{{Conjecture}}
\newtheorem{example}[theorem]{{Example}}

\newtheorem{definition}[theorem]{{Definition}}
\newtheorem{claim}[theorem]{{Claim}}
\theoremstyle{remark}
\newtheorem{remark}[theorem]{Remark}

\makeatletter\@addtoreset{equation}{section} \makeatother

\def\bclaim{\begin{claim}}
\def\eclaim{\end{claim}}
\def\bdefin{\begin{definition}}
\def\edefin{\end{definition}}
\def\bcor{\begin{corollary}}
\def\ecor{\end{corollary}}
\def\bthm{\begin{theorem}}
\def\ethm{\end{theorem}}
\def\bconj{\begin{conjecture}}
\def\econj{\end{conjecture}}
\def\blem{\begin{lemma}}
\def\elem{\end{lemma}}
\def\blemma{\begin{lemma}}
\def\elemma{\end{lemma}}
\def\bprop{\begin{prop}}
\def\eprop{\end{prop}}
\def\bremark{\begin{remark}}
\def\eremark{\end{remark}}
\def\bhyp{\begin{hypothesis}}
\def\ehyp{\end{hypothesis}}
\def\bnot{\begin{notation}}
\def\enot{\end{notation}}
\def\bexample{\begin{example}}
\def\eexample{\end{example}}
\def\lb{\label}

 \def\be{\beta}
\def\o{\omega} \def\O{\Omega}
 
 \def\eps{\epsilon}

\newcommand{\dbar}{\overline\partial}
\newcommand{\ddbar}{\partial\dbar}

\newcommand{\vol}{{\operatorname{vol}}}

\def\max{{\operatorname{max}}}

\def\KE{K\"ahler--Einstein }

\def\Ric{\hbox{\rm Ric}\,}

\def\h#1{\hbox{#1}}

\def\strutdepth{\dp\strutbox}
\def\specialstar{\vtop to \strutdepth{
    \baselineskip\strutdepth
    \vss\llap{$\star$\ \ \ \ \ \ \ \ \  }\null}}
\def\marginalstar{\strut\vadjust{\kern-\strutdepth\specialstar}}

\def\marginal#1{\strut\vadjust{\kern-\strutdepth
    {\vtop to \strutdepth{
    \baselineskip\strutdepth
    \vss\llap{{ \small #1 }}\null}
    }}
    }

\def\dis{\displaystyle}

\def\q{\quad} \def\qq{\qquad}

\def\ra{\rightarrow}

\def\w{\wedge}
\def\i{\sqrt{-1}}

\let\s=\sigma

\def\sm{\setminus}

\def\exc{\operatorname{exc}}
\def\lct{\operatorname{lct}}
\def\D{\Delta}

 \newcommand{\RR}{\mathbb{R}}
\newcommand{\QQ}{\mathbb{Q}} 
 \newcommand{\NN}{{\mathbb N}}

\def\la{\lambda}
\def\beq{\begin{equation}}
\def\eeq{\end{equation}}
\def\bpf{\begin{proof}}
\def\epf{\end{proof}}

\def\eaeq{\end{aligned}}
\def\baeq{\begin{aligned}}
\def\mult{\operatorname{mult}}

\def\lb{\label}
 
\def\KE{K\"ahler--Einstein } 
\def\KEE{K\"ahler--Einstein edge }

\def\supp{{\operatorname{supp}}}
\def\ord{\operatorname{ord}}
\def\mult{\operatorname{mult}}
\def\bi{\bibitem}

\def\glct{\operatorname{glct}}
\def\blct{\operatorname{blct}}
\def\simq{\sim_\mathbb{Q}}

\newcommand\blfootnote[1]{%
  \begingroup
  \renewcommand\thefootnote{}\footnote{#1}%
  \addtocounter{footnote}{-1}%
  \endgroup
}

\author{Ivan A. Cheltsov, Yanir A. Rubinstein, Kewei Zhang}

\title{
Basis log canonical thresholds, 
local intersection estimates, \\ and 
asymptotically log del Pezzo surfaces}

\begin{document}

\maketitle

\begin{abstract}

The purpose of this article is to develop techniques for estimating
 basis log canonical thresholds on logarithmic surfaces.
To that end, we develop new local intersection estimates
that imply log canonicity.
Our main motivation and application is to  show  
the existence of K\"ahler--Einstein edge metrics on
all but finitely many families of
asymptotically log del Pezzo surfaces, partially confirming a conjecture of
two of us.
In an appendix we show that the basis log canonical threshold 
of Fujita--Odaka coincides with the greatest lower Ricci bound invariant of Tian.

\end{abstract}

\blfootnote{Research supported by NSF grant DMS-1515703 and 
the China Scholarship Council award 201706010020. We thank C. Li for comments on an earlier version.}

\section{Introduction}

\subsection{Estimating basis log canonical thresholds}
\label{}

Global and local log canonical thresholds naturally play a crucial r\^ole in algebraic geometry. 
For instance, Shokurov's conjecture \cite{Shokurov} on the ascending chain condition for local log canonical thresholds 
(proved in \cite{HMX}) implies the inductive step in termination of higher-dimensional log flips \cite{Birkar2007}.
Likewise, Birkar's boundedness results \cite{Birkar} for global log canonical threshold imply the Borisov--Alexeev--Borisov conjecture in all dimensions:  
the set of Fano varieties of dimension $d$ with $\epsilon$-log canonical singularities forms a bounded family for given $d\in\NN$ and $\epsilon>0$.
This conjecture implies that the birational automorphism group of any rationally connected variety is Jordan \cite{ProkhorovShramov},
so that, in particular,  all Cremona groups are Jordan (in dimension $2$ this was proved by Serre in \cite{Serre}).
Global log canonical thresholds are used to prove irrationality of Fano varieties \cite{Pukhlikov,CheltsovAdvances}
the absence of non-trivial fiber-wise birational maps between Mori fiber spaces \cite{CheltsovPark,CheltsovAJM,CheltsovSbornik},
the uniqueness of a Koll\'ar component of a Kawamata log terminal singularity \cite{CheltsovShramovGT,CheltsovShramovCrelle,CheltsovParkShramov},
and non-conjugacy of finite subgroups in Cremona groups \cite{CheltsovToolBox},
Moreover, a combination of results about global and local log canonical thresholds of del Pezzo surfaces
helped to answer an old standing open question in affine geometry \cite{CheltsovParkWon},
and make a first step towards Gizatullin's conjecture about automorphisms of the affine complements to Fano hypersurfaces \cite{CheltsovDuboulozPark}. 

About a decade ago, it was realized \cite{CSD} that global log canonical thresholds (glcts) are the algebraic analogues of Tian's alpha invariants \cite{T87} 
that are central in the study of K\"ahler--Einstein (KE) metrics. This paved the way to
using a wide range of algebraic tools
to prove existence of  KE metrics via Tian's theorem
that stipulates that an estimate on Tian's invariant guarantees the existence
of such a metric \cite{CheltsovAdvances,CheltsovToolBox,CheltsovAJM,CheltsovParkShramovJGA,CheltsovKosta,CheltsovShramovMIAN,CheltsovParkWonMZ}.
While this has been
arguably the most fruitful method for finding new KE metrics,
a sticking point with this approach  has been that
Tian's theorem only provides a sufficient condition for the existence of KE metrics.
Recently, this has been remedied by Fujita--Odaka \cite{FujitaOdaka}
that introduced a new invariant, that we refer to as the
basis log canonical threshold (blct) reminiscent of the
global log canonical threshold (the blct has also
been referred to as the delta invariant and the stability
threshold in the literature, see Definition \ref{definition:delta} below for detailed references).
The advantage of estimating this modified threshold is that
it provides a necessary and sufficient condition for K-stability, which in turn is equivalent to the existence of KE metrics \cite{CDS,Tian2015}.
In fact, in an appendix we show that the algebraic invariant blct coincides with an analytic invariant studied by Tian almost thirty years ago, namely, the greatest Ricci lower bound invariant.

However, while there are many techniques for estimating glcts, at the moment
rather little is known about how to actually estimate blcts.
Recently, an important first step in this direction was taken  by Park--Won
\cite{PW} who developed algebraic methods for estimating blcts in dimension 2. 
First, they explicitly computed blct of $\mathbb{P}^2$ similar to what was  done later 
for all toric Fano varieties by Blum--Jonsson \cite{BJ17}.
Then Park and Won used the key fact that every two-dimensional Fano variety different from $\mathbb{P}^1\times\mathbb{P}^1$
can  be (non-canonically) obtained from $\mathbb{P}^2$ by blowing up $\le 8$ points in general position.
This allowed them to use their toric computations to estimate blcts from above.
The Park--Won approach is quite different from computations of glcts in the literature \cite{CheltsovToolBox}
where all estimates are done using only intrinsic geometry of the surface.
Moreover, it is not clear how to adapt toric-type 
computations as in \cite{PW} to our logarithmic setting which also includes
a boundary divisor. Finally, it is also not clear how to adapt their 
method to higher dimensions, simply because blow-ups of projective space or other higher-dimensional toric variety 
very rarely have ample anticanonical class.
While we do not tackle this here, we believe that
it should be possible to extend our method 
to dimension three.

In this article we develop intrinsic techniques
for estimating blcts in dimension 2.
The methods we develop are of a geometric nature and involve new criteria for log canonicity in terms of local intersection numbers, that
we believe are of independent interest. As we show in
a sequel, in the setting considered by Park--Won our methods yield
stronger estimates with perhaps more geometric proofs.
Moreover, in this article we use our new log canonicity criteria,
coupled with vanishing order estimates for basis divisors
to prove the log K-stability of an important family of logarithmic surfaces.
This partially resolves a conjecture of two of us that we now turn to describe.

\subsection{\KEE metrics}
\label{KEESec}

Smooth \KE (KE) metrics have been studied for over 80 years, with intriguing relations to algebraic geometry emerging over the last 30 years. More recently, motivated by suggestions of Tsuji, Tian, and Donaldson,
singular \KE metrics called \KEE (KEE) metrics have been intensely studied, mainly as a tool for understanding smooth \KE metrics.
KEE metrics are a natural generalization of KE metrics: they are smooth metrics on the complement of a divisor, and have a conical singularity of angle $2\pi\be$
transverse to that `complex edge' (i.e.,
the metric as being `bent` at an angle $2\pi\be$ along the divisor).
They tie naturally to the study of log pairs in algebraic geometry.
When $\be=1$, of course, a KEE metric is just
an ordinary KE metric that extends
smoothly across the divisor, and so understanding
existence of KEE metrics as well as their asymptotics
near the divisor \cite{JMR} as well as
the limit  $\be\ra 1$ \cite{CDS,Tian2015}
has attracted much work;
we
refer to the survey \cite{R14} for a thorough discussion and many more references.

In \cite{CR}, two of us initiated a program whose aim is to understand the behavior in
the other extreme when the \emph{cone angle $\beta$ goes to zero} consisting of:
\begin{itemize}
\item [(a)] Classifying all triples $(X,D,\be)$ satisfying
the necessary cohomological condition \eqref{cohomEq} for \emph{sufficiently small $\beta$};
\item [(b)] Obtaining a condition
equivalent to existence of KEE metrics for such triples;
\item[(c)] Understanding the limit, when such exists, of these KEE metrics as $\be$
tends to zero.
\end{itemize}
The cohomological condition alluded to in (a),
\beq\label{cohomEq}
-K_X-(1-\be)D \h{\ \ is\ $\mu$ times an ample class, for some $\mu\in\RR$},
\eeq
is also the necessary and sufficient condition for (b)
if $\mu\le0$ \cite[Theorem 2]{JMR}; moreover, a classification (i.e., part (a))
is essentially impossible when  $\mu\le0$
\cite{DiCerbo},
\cite[\S8]{R14},
and so we will restrict our attention in (a)--(b) exclusively to the case $\mu>0$, that we have previously called the {\it asymptotically log Fano regime}  \cite [Definition 1.1] {CR}.

Our previous work accomplished (a) in dimension 2, providing a complete classification  \cite [Theorem 2.1] {CR}. Furthermore, we also obtained the ``necessary" portion of (b) \cite{CR,CR2}, and this was extended to higher dimensions by Fujita \cite{Fujita}.
The purpose of this article is to complete the ``sufficient" portion of (b) in dimension 2 in all but finitely  many (in fact, all but 6) of the (infinite list
of) cases  classified in  \cite{CR}.

\subsection{The Calabi problem for asymptotically log Fano varieties}

A special class of
asymptotically log Fano varieties is as follows.
This is a special case of   \cite [Definition 1.1] {CR}.

\begin{definition}
\label{definition:log-Fano} We say that a pair $(X,D)$ consisting
of a smooth
projective variety $X$ and a
smooth irreducible divisor $D$ on $X$ is {\it asymptotically log Fano} if
the divisor $-K_X-(1-\be)D$ is ample for
sufficiently small $\beta\in(0,1]$.
\end{definition}

This definition contains the class of smooth Fano varieties ($D=0$) as well as the classical notion of a smooth log Fano pair due to Maeda ($\be=0$) \cite{Maeda}.

One can show
using a result of Kawamata--Shokurov
that
if $(X,D)$ is asymptotically log Fano
then $|k(K_{X}+D)|$ (for some $k\in\NN$) is
free from base points and gives a morphism
\cite[\S1]{CR}
$$
\eta\colon X\to Z.
$$
The following conjecture, posed in our earlier work, gives a rather complete picture
concerning (b) when $D$ is smooth.

\begin{conjecture}{\rm\cite[Conjecture 1.11]{CR}}
\label{conjecture:big} Suppose that $(X,D)$ is an asymptotically log
Fano manifold with $D$ smooth and irreducible.
There exist KEE metrics with angle $2\pi\beta$ along $D$ for all
sufficiently small $\beta$ if and only if $\eta$ is not birational.
\end{conjecture}

This conjecture stipulates that the existence problem for KEE
metrics in the small angle regime boils down to a simple
birationality criterion. 
In fact, this amounts to computing a single
intersection number, i.e.,  
checking whether
$$
(K_X+D)^n=0.
$$
This would be a rather far-reaching simplification
as compared to checking the much harder condition of
log K-stability
that involves, in theory,
computing the Futaki invariant of an infinite number of
log test configurations, or else estimating the blct
which involves, in theory, estimation of singularities of pairs
that may occur after an unbounded number of blow-ups.

\subsection{The Calabi problem for asymptotically log del Pezzo surfaces}
\label{}

Following \cite{CR,R14} we will refer to understanding (b) as the Calabi problem for asymptotically log Fano varieties.
This article makes an important step towards solving this problem in dimension 2,
where Fano varieties are commonly called del Pezzo surfaces.
To explain this, let us recall what is already known about this problem from our previous work.

\subsubsection{The big case}
\label{}

The necessary direction of Conjecture \ref{conjecture:big} in dimension 2 is known as we now recall.

According to  \cite [Theorems 1.4,2.1] {CR} asymptotically log del Pezzo pairs $(X,D)$ (with $D$ smooth and irreducible) for which $-K_X-D$ is big are as follows.
Either $(X,D)$ is one of the five Maeda pairs (i.e., with $-K_X-D$ ample):
\begin{itemize}
\item
$\mathrm{(I.1B}):=(\mathbb{P}^2$, smooth conic),%
\item
$\mathrm{(I.1C}):=(\mathbb{P}^2$, line),%
\item
$\mathrm{(I.2.n}):=(\mathbb{F}_{n},Z_n)$ (where, for any $n\ge 0$, $\mathbb{F}_{n}$ is the Hirzebruch surface containing a curve $Z_n$ whose self intersection is $-n$ and fiber $F$
whose self intersection is $0$),
\item
$\mathrm{(I.3B}):=(\mathbb{F}_1$, smooth element of $|Z_1+F|$),
\item
$\mathrm{(I.4C}):=(\mathbb{P}^1\times\mathbb{P}^1$, smooth bi-degree $(1,1)$ curve),
\end{itemize}
or else $(X,D)$ is obtained from one of the five Maeda surfaces $(X_M,D_M)$ as follows:
$X$ is the blow-up of $X_M$ at any number of distinct points on $D_M$, and $D$
is the proper transform of $D_M$.

According to Conjecture \ref{conjecture:big} none of these pairs should admit KEE metrics with small angles. This was
verified in a unified
manner using
flop slope stability
\cite[Theorem 1.6]{CR2}, but
can also be obtained
as follows:
for  $\mathrm{(I.1B})$
 and $\mathrm{(I.4C})$
 \cite[Example 3.12]{LiSun},
for $\mathrm{(I.1C}):=(\mathbb{P}^2$, line)%
and $\mathrm{(I.2.n})$
as well as their blow-ups this is a consequence of the Matsushima theorem for edge metrics \cite[Theorem 1.12, Proposition 7.1]{CR}, for
$\mathrm{(I.3B})$
 \cite[Example 2.8]{CR2}, for the blow-ups of
$\mathrm{(I.1B}), \mathrm{(I.3B})$ and $\mathrm{(I.4C})$  \cite[Proposition 5.2]{CR2}.

\subsubsection{The non-big case}
\label{}

The harder sufficient direction of Conjecture \ref{conjecture:big} in dimension 2 is still partly open. Let us recall the state-of-the-art.

According to  \cite [Theorems 1.4,2.1] {CR} asymptotically log del Pezzo pairs $(X,D)$ (with $D$ smooth) for which $(K_X+D)^2=0$ are as follows:
\begin{itemize}
\item
$
(X,D)$ is del Pezzo with $D$ a smooth anti-canonical curve,
\item
$
\mathrm{(I.3A}):=(\mathbb{F}_1$, smooth element of $|2(Z_1+F)|$), %
\item
$
\mathrm{(I.4B}):=(\mathbb{P}^1\times\mathbb{P}^1$, smooth bi-degree $(2,1)$ curve), %
\item
 $
\mathrm{(I.9B.m}):=(X,D)$ with $X$ a blow-up of $\mathbb{P}^1\times\mathbb{P}^1$
at $m$ distinct points on a smooth bi-degree $(2,1)$ curve
with no two of them on a single curve of bi-degree $(0,1)$,
and $D$ is the proper transform of the bi-degree $(2,1)$ curve.
\end{itemize}

According to Conjecture \ref{conjecture:big} all of these pairs should admit KEE metrics with small angles. The following cases are known:
del Pezzo with a smooth anti-canonical curve
\cite[Corollary 1]{JMR}, 
$\mathrm{(I.3A})$ \cite[Proposition 7.5]{CR},
$\mathrm{(I.4B})$  \cite[Proposition 7.4]{CR}. Thus, the only remaining cases are
$
\mathrm{(I.9B.m}), m\ge1$.

\subsection{Main result}
\label{}

In this article, we treat all but finitely many of the remaining open cases
$\mathrm{(I.9B.m})
$:

\begin{theorem}
\label{main}
The log Fano pairs $
\mathrm{(I.9B.m}), m\ge 7$ are uniformly log K-stable for all sufficiently small $\be>0$.
\end{theorem}

Indeed, recent  results show that
log K-stability implies the existence of a KEE metric
on a given log Fano pair
 \cite{CDS,Tian2015,TianWang}.
The finitely-many remaining cases 
$\mathrm{(I.9B.m}), 1\le m\le 6$ require a different
approach and will be discussed elsewhere
(although we omit the details, the techniques of this article can also be used to show the cases
$\mathrm{(I.9B.5)}$
and
$\mathrm{(I.9B.6)}$
are log K-semistable).

In the course of the proof we develop new local intersection criteria for showing log-canonicity on a surface---see
Section \ref{multSec}. We believe these are of substantial interest independently of their application to proving Theorem \ref{main}.
Sections \ref{section:many-blow-ups}--\ref{blctSec} are concerned
with the proof of Theorem \ref{main}. In \S\ref{section:many-blow-ups} we estimate the vanishing order of divisors on the logarithmic
surfaces $\mathrm{(I.9B.m)}$, while in \S\ref{blctSec} we use
these estimates together with the criteria of \S\ref{multSec}
to estimate the basis log canonical threshold of 
the logarithmic
surfaces $\mathrm{(I.9B.m)}$. The article concludes
with an appendix that identifies the
basis log canonical threshold of Fano manifolds with
Tian's greatest Ricci lower bound.
After this paper was first posted, we were informed that
Berman--Boucksom--Jonsson also obtained 
Theorem \ref{theorem:delta-equal-R} independently
and that Blum--Liu obtained a variant of Lemma 
\ref{lemma:log-delta-converge-to-delta} \cite{BL18}.

\section{Preliminaries}

 In this section---except in the last lemma where we 
specialize to surfaces ($n=2$)---we let $X$ be a complex algebraic variety of complex dimension $n$.

\subsection{Log pairs }

Given a proper birational morphism $\pi:Y\ra X$, we define
the exceptional set of $\pi$ to be the smallest subset $\exc(\pi)\subset Y$,
such that $\pi:Y\sm\exc(\pi)\ra X\sm \pi(\exc(\pi))$ is an isomorphism.

A log resolution of $(X,\Delta)$ is a proper birational morphism
$\pi:Y\ra X$ such that $\pi^{-1}(\Delta)\cup\{\exc(\pi)\}$
is divisor with simple normal crossing (snc) support.
Log resolutions exist for all the pairs we will consider in this article,
by Hironaka's theorem.

Assume that \mbox{$K_{X}+\D$} is a~$\mathbb{Q}$-Cartier
divisor.
Given a log resolution of $(X,\Delta)$, write
$$
\pi^\star (K_X+\Delta)=K_Y+\tilde \Delta +\sum e_i E_i,
$$
where
$\tilde \D$ denotes the proper transform of $\Delta$, and
where $\exc(\pi)=\cup E_i$, and $E_i$ are irreducible codimension one subvarieties.
Also, assume $\Delta=\sum \delta_i\Delta_i$, with $\Delta_i$ irreducible codimension one subvarieties,
so $\tilde \Delta=\sum \delta_i\tilde \Delta_i$.
Singularities of pairs can be measured as follows.

\begin{definition}
Let $Z\subset X$ be a subvariety.
A pair $(X,\Delta)$ has at most log canonical (lc) singularities (or klt singularities, respectively)
along $Z$ if $e_i,\delta_j\le 1$ for every $i$ (or if $e_i,\delta_j< 1$ for every $i$, respectively) such that
$\pi(E_i)\cap Z\not=\emptyset$ and every $j$ such that $\Delta_j\cap Z\not=\emptyset$.
\end{definition}

On a normal variety,
an effective $\mathbb{Q}$-divisor $D$
is a formal linear combination with coefficients in $\mathbb{Q}_+$
of prime divisors.
Thus, given such a $D$ and
a prime divisor $F$, one has $D=aF+\D$,
for some $a\in\QQ_+$ and $\D$
is an effective $\mathbb{Q}$-divisor
with $F\not\subset \supp\D$.
The number $a$ is called the {\it vanishing order of $D$ along $F$}, denoted
\begin{equation*}
\label{ordEq}
\ord_FD.
\end{equation*}

\subsection{Log canonical thresholds}
\lb{lctsubsec}

\begin{definition}
Let $Z\subset X$ be a subvariety
and let $\Delta$ be 
Cartier $\mathbb{Q}$-divisor on $X$.
The log canonical threshold of the pair $(X,\Delta)$ along $Z$ is
$$
\lct_Z(X,\Delta):=\sup\{\lambda\,:\, (X,\lambda \Delta) \h{\ is log-canonical along $Z$}\}.
$$
Set $\lct(X,\Delta):=\lct_X(X,\Delta)$.
\end{definition}

Let $(X,B)$ be a~klt log pair. Let $D$ be an effective $\mathbb{Q}$-Cartier divisor on the~variety $X$. Recall that the log canonical threshold of the~boundary
$D$  is the number
$$
\mathrm{lct}\big(X,B;D\big)=\mathrm{sup}\left\{c
\,:\, \mathrm{the~pair}\  \big(X, B+c D\big)\ \mathrm{is\ log\ canonical}\right\}.%
$$
Let $H$ be an ample $\mathbb{Q}$-divisor on $X$, and let $[H]$ be
the class of the divisor $H$ in
$\mathrm{Pic}(X)\otimes\mathbb{Q}$.

A fact we will use over and over again is that the property of being lc or klt
is preserved under blow-ups, and therefore can be checked either upstairs on $Y$ or downstairs
on $X$ \cite[Lemma 3.10]{KollarNotes}.
When $n=2$ this becomes quite concrete:
let $\pi:\tilde{S}\rightarrow S$ be the blow up of
a point $p\in S$ and let $E:=\pi^{-1}(p)$.
Denote by $\tilde{\Delta}$ the proper transform of $\Delta$
under $\pi$. Then
the log pair $(S,\Delta)$ is lc/klt at $p$
if and only if
$(\tilde{S}, \tilde{\Delta}+(\mult_P\Delta-1)E)$ is lc/klt
along $E$  \cite[Remark 2.6]{CheltsovToolBox}.

\begin{definition}
\label{definition:global-threshold} The
{\rm global log canonical threshold}
of the
 pair $(X,B)$ with respect to $[H]$ is
$$
\glct(X,B,[H]):=
\mathrm{sup}\big\{ 
c>0\,:\, (X, B+cD) \h{\ is lc 
for every $D\sim_{\mathbb{Q}} H$}
\big\}.
$$
\end{definition}

\subsection{The basis log canonical threshold
}

In this part we collect some known results
about a basis-type invariant for log pairs
due to Fujita--Odaka \cite{FujitaOdaka},
see also
\cite{BJ17,CP-delta-and-stability-for-log-case}.

Let $L$ be an ample 
$\mathbb{Q}$-divisor in $X$.
For any $k\in\mathbb{N}$ such that $kL$ is Cartier, let
$$
d_k:=\text{dim}_\mathbb{C}H^0(X,kL)>0.
$$
In this article, whenever we mention
multiples $kL$ of $L$

\medskip

\centerline{\sl
we will always assume (implicitly) that $k$
is such an integer.}

\begin{definition}
\label{definition:basisdiv}
We say that $D\sim_{\QQ}L$ is a {\rm basis divisor} if for
some $k\in\NN$,
$$
D=\frac{1}{kd_k}\sum_{i=0}^{d_k}(s_i),
$$
where
$s_1,...,s_{d_k}$ is a basis of $H^0(X,kL)$, and where
$
(s_i)
$
is the divisor cut out by $s_i$.
We also say that $D$ is the
$k$-basis divisor associated to
the basis $\{s_i\}_{i=1}^{d_k}$.
\end{definition}

 The following definition is 
due to Fujita--Odaka \cite[Definition 0.2]{FujitaOdaka}, 
extended to the logarithmic setting by Fujita 
\cite[Definition 5.4]{FujitaOpeness} 
(Fujita's
definition can be shown to equal a logarithmic version of the original definition of Fujita--Odaka, see \cite{BJ17,CP-delta-and-stability-for-log-case}) who denoted it $\delta(X,B)$,
and Codogni--Patakfalvi \cite[Definition 4.3]{CP-delta-and-stability-for-log-case} who denoted it $\delta(X,B;L)$.
It roughly amounts to replacing
``$D$ effective $\QQ$-divisor"
by
``$D$ basis divisor"
in Definition \ref{definition:global-threshold}.
So, it yields an  invariant larger than glct, albeit one that is significantly more difficult to compute.

\begin{definition}
\label{definition:delta}
Let $(X,B)$ be a klt log pair.
The {\rm basis log canonical threshold}
of the
 pair $(X,B)$ with respect to $L$ is
$$
\blct_\infty(X,B,L):=
\limsup_k \blct_k(X,B,L),
$$
where
$
\blct_k(X,B,L):=
\sup
\{c>0\,:\, (X,B+cD)\text{ is lc for any $k$-basis
divisor }D\sim_{\QQ}L\}.
$

\end{definition}

Estimating the invariant $\blct_\infty(X,B,L)$ is
of interest since it coincides with
an analytic invariant related to Ricci curvature (see
Theorem \ref{theorem:delta-equal-R} below)
and also
due to the following theorem that 
follows from the work of Fujita--Odaka \cite{FujitaOdaka}, Fujita \cite{Fujita-criterion-for-K-stability}, Li \cite{ChiLi-Criterion},
Blum--Jonsson \cite{BJ17} (see also \cite[Corollary 4.8]{CP-delta-and-stability-for-log-case}).
For the precise definition of uniform log K-stability,
we refer the reader to \cite[Definition 8.1]{BHJ17}.

\begin{theorem}
\label{FujThm}
The triple $(X,\D,-K_X-\D)$ is uniformly log K-stable
if\hfill\break 
$\blct_\infty(X,\D,-K_X-\D)>1$. 
\end{theorem}

\def\pis{\pi^\star }

\subsection{Volume estimates on the order of vanishing }
\label{}
In most of this subsection  we follow closely \cite{FujitaOdaka}.
Estimating log canonical thresholds naturally involves estimating from above orders of vanishing along divisors, oftentimes upstairs on a resolution of the original log pair
(recall \S\ref{lctsubsec}).
A crude upper bound on $\ord_F(\pis D)$ is the
pseudoeffective threshold of the divisor $\pis L$ with respect to the curve $F$,
\beq 
\lb{tauEq}
\tau(\pis L,F)=\mathrm{sup}\{\lambda\,:\,
\pis L-\lambda F\ \text{is effective}\},
\eeq
since $D\sim_\QQ L$ and $\pis D=\ord_F(\pis D)F+\Delta$,
with $\Delta$ an effective $\mathbb{Q}$-divisor,
whose support does not contain the curve $F$.
A better estimate can be obtained by using a quantized version
of the pseudoeffective threshold,
$$
\tau_k(\pis L,F):=\max\{x\in\NN \,:\, H^0(Y,k\pi^\star L-xF)\neq0\}.
$$
When no confusion arises we will often abbreviate these two invariants by
$\tau$ and $\tau_k$.
Note that,
\begin{equation}
\begin{aligned}
\label{tauktaueq}
\limsup_k\tau_k(\pis L,F)/k=\tau(\pis L,F),
\end{aligned}
\end{equation}
as trivially
$\tau_k/k\le \tau$ for every $k$ (let
$s\in H^0(Y,k\pi^\star L-xF)$, then $(s)/k \simq \pi^\star L-\frac xkF$ is effective),
while if $\pi^\star L-\frac xF$ is big,
then $H^0(Y,i_k\pi^\star L-i_kxF)\not=0$
for a increasing sequence of integers $\{i_k\}$
(rememeber we are working with $\QQ$-divisors)
so $\tau_{i_k}\ge i_kx$, i.e., $\limsup_k\tau_k/k\ge x$,
and now let $x\ra\tau$.

\begin{lemma}
\lb{lemma:FujLem}
Let $\pi:Y\ra X$ be a log resolution of $(X,\Delta)$, and
let $F$ be a prime divisor in $Y$.
Let $D\sim_{\QQ} L$ be a $k$-basis  divisor.
Then
$$
\ord_F(\pis D)\leq\frac1{kd_k}{\sum_{b=1}^{\tau_k(\pis L,F)}h^0(Y,k\pis L-bF)},
$$
and equality is attained for an appropriate choice of basis.
\end{lemma}

\begin{proof}
For completeness, we provide the proof that
can be easily extracted from \cite[Lemma 2.2]{FujitaOdaka}.
Fix $Y$ and $F$ as in the statement.
Filter $H^0(Y,k\pi^\star L)$ in increasing order of vanishing along $F$,
$$
H^0(Y,k\pi^\star L)\supseteq H^0(Y,k\pi^\star L-F)\supseteq
\ldots \supseteq H^0(Y,k\pi^\star L-\tau_kF)\supset
H^0(Y,k\pi^\star L-\tau_kF-F)=\{0\}.
$$
Now, fix a basis $s_1,...,s_{d_k}$ of $H^0(X,kL)$,
and let $D\sim_{\QQ} L$ be the associated $k$-basis divisor (recall
Definition \ref{definition:basisdiv}).
For each $b\in\{0,\ldots,\tau_k+1\}$ suppose that exactly $i(b)$
of the
sections $s_1\circ \pi,...,s_{d_k}\circ \pi$ are elements in $H^0(Y,k\pi^\star L-bF)$.
Note that $i(0)=h^0(X,kL)$ and $i(\tau_k+1)=0$, and denoting
$$
h^0(Y,k\pi^\star L-bF):=
\dim H^0(Y,k\pi^\star L-bF),
$$
of course we have $i(b)\le h^0(Y,k\pi^\star L-bF)$.
Then,
$$
\ord_F(\pis D)
=
\frac{\sum_{b=1}^{\tau_k}i(b)}{kd_k}
\le
\frac{\sum_{b=1}^{\tau_k}h^0(Y,k\pi^\star L-bF)}{kd_k}
.$$
So we get
$$
\text{ord}_F(\pis D)\leq\frac{\sum_{b=1}^{\tau_k}h^0(kL-bF)}{kd_k}.
$$
Moreover, we may choose a basis
$\tilde s_1,...,\tilde s_{d_k}$ of $H^0(X,kL)$ as follows to obtain
for the associated $k$-basis divisor $\tilde D$,
$$
\text{ord}_F
(\pi^\star\tilde{D})
=\frac{\sum_{b=1}^{\tau_k}h^0(kL-bF)}{kd_k},
$$
as follows: let
$\tilde s_1,\ldots,\tilde s_{h^0(Y,k\pi^\star L-k\tau_kF)}$
be a basis for $H^0(Y,k\pi^\star L-\tau_kF)$;
thus, $i(\tau_k(F))=h^0(Y,k\pi^\star L-\tau_kF)$.
Next, choose the following $h^0(Y,k\pi^\star L-\tau_kF+F)-i(\tau_k)$
$\tilde s_i$'s
to complete the sections from the first step to a basis for
$H^0(Y,k\pi^\star L-\tau_kF+F)$. Thus,
$i(\tau_k-1)=h^0(Y,k\pi^\star L-\tau_kF+F)$.
By induction, we see that $i(b)=h^0(Y,k\pi^\star L-bF)$
for each $b$, as desired.
\end{proof}

Asymptotically, we may estimate the sum in Lemma \ref{lemma:FujLem}
using volumes. Let us first recall some basic facts about volumes,
following
\cite{Laz-positivity-in-AG,LM-OkounkovBodyTheory}.

\begin{definition}
\label{definition:vol}

Let $D$ be a Cartier divisor on $X$. The volume of $D$ is defined by
$$
\vol(D):=\limsup_k \frac{h^0(X,kD)}{k^n/n!}.
$$
\end{definition}

In fact, one may replace the limsup by a limit
\cite[Example 11.4.7]{Laz-positivity-in-AG},
and
by rescaling and continuity
\cite[Corollary 2.2.45]{Laz-positivity-in-AG}
$\vol(D)$ makes sense for
any $\mathbb{R}$-Cartier divisor $D$. Also, the volume function is invariant under pull-back by a birational morphism, i.e., $\vol(\pi^\star D)=\vol(D)$.
Finally,
\begin{equation}
\begin{aligned}
\label{nefvol}
\h{when the divisor $D$ is nef
(i.e., a limit of ample divisors)
then $\vol(D)=D^n$.}
\end{aligned}
\end{equation}

\begin{cor}
\lb{FujCor}
Let $\pi:Y\ra X$ be a log resolution of $(X,\Delta)$, and
let $F$ be a prime divisor in $Y$.
Let $D\sim_{\QQ} L$ be a $k$-basis divisor.
Then,
$$
\text{ord}_F(\pis D)\leq
\frac 1{L^n}{\int_0^{\tau(\pis L,F)}\vol(\pi^\star L-xF)}dx+\eps_k,
$$
with
$\lim_k\eps_k=0$.
\end{cor}

\begin{proof}
This result is probably standard (see, e.g., \cite[Lemma 4.7]{FujitaVolume}),
but since it plays an important
r\^ole in this article let us sketch a proof.
By  Riemann--Roch asymptotics,
$L^n/n!=d_k/k^n+O(1/k)$ \cite[1.4.41]{Laz-positivity-in-AG}.
Thus,
$$\frac{k^n}{n!d_k}=\frac{1}{L^n}+O(1/k)$$
Define a decreasing step function by
$$
f_k(x):=\frac{h^0(Y,k\pis L-\lfloor kx\rfloor F)}{k^n/n!},\qq x\in[0,\infty).
$$
Then by Okounkov body theory for filtrated linear series
\cite[Lemma 1.6]{BC09},
\cite[Theorem 2.13]{LM-OkounkovBodyTheory},
\cite[Theorem 5.3]{BHJ17},
$
f_k(x)=\vol(\pis L-xF)+\eps_k,
$
with
$\lim_k\eps_k=0$;
thus,
$$
\frac{k^nf_k(x)}{n!d_k}=\frac{\vol(\pis L-xF)}{L^n}+\eps_k.
$$
In other words, as $k\rightarrow\infty$, the function $\frac{k^nf_k(x)}{n!d_k}$ converges pointwise to $\frac{\vol(\pis L-xF)}{L^n}$ for $x\in[0,\infty)$.
Finally, using Lemma \ref{lemma:FujLem} and dominated convergence, we see that
\begin{equation*}
\begin{aligned}
\text{ord}_F(\pis D)&\leq
\frac{\sum_{b=1}^{\tau_k(\pis L,F)}h^0(Y,k\pis L-bF)}{kd_k}\\
&=\int_0^{\frac{\tau_k(\pis L,F)}{k}}\frac{k^nf_k(x)}{n!d_k}dx+\eps_k\\
&=\frac{1}{L^n}\int_0^{\tau(\pis L,F)}\vol(\pis L-xF)dx+\eps_k,
\end{aligned}
\end{equation*}
($\eps_k$ can change from line to line
as long as
$\lim_k\eps_k=0$)
where we used   \eqref{tauktaueq}
(although it is actually enough to use
$\limsup_k\tau_k(\pis L,F)/k\ge \tau(\pis L,F)$ as $\vol(\pis L-xF)=0$
for $x>\tau(\pis L,F)$, i.e., it is enough to integrate until $\tau(\pis L,F)$).
\end{proof}

The following lemma is handy when computing the volume of divisors on a surface.
\begin{lemma}
\label{lemma:vol-replacement-of-line-bundle}
Let $B$ be a $\mathbb{Q}$-Cartier divisor on a surface $S$
and let $Z$ be a curve in $S$ with $Z^2<0$ and $B.Z\le0$.
Then, $$\vol(B)=\vol\Big(B-\frac{B.Z}{Z^2} Z\Big).$$
\end{lemma}

\begin{proof}
Take $k\in\NN$ so that $kB$ is Cartier and let
$D\in|kB|.$
Decompose, $D=Z\ord_ZD+\Delta.$
Then,
$$
kB.Z=D.Z=Z^2\ord_ZD+\Delta.Z\ge Z^2\ord_ZD,$$
and as  $Z^2<0$ this yields
$$\ord_ZD\ge  k\frac{B.Z}{Z^2},$$
and the right hand side is non-negative as $B.Z\le0$.
Since $D$ was any element of $|kB|$, we have shown that
$$
h^0(S,kB)=h^0\Big(S,kB-k\frac{B.Z}{Z^2}Z\Big).$$
By Definition \ref{definition:vol} we are done.
\end{proof}

\section{Local intersection estimates on surfaces}
\label{multSec}

In this section we derive  new criteria for log canonicity in terms of local intersection estimates.

Let $\mathcal{O}_{p}$ be the local ring of germs of holomorphic
functions defined in some neighborhood of $p$.

\begin{definition}
\label{definition-local}
Let $C_1$ and $C_2$ be two irreducible curves on a surface $S$. Suppose that $C_1$ and $C_2$ intersect at a smooth point $p\in S$. Then the local intersection number of $C_1$ and $C_2$ at the point $p$ is defined by
$$(C_1.C_2)_p:=\dim_{\mathbb{C}} \mathcal{O}_{p}/(f_1,f_2),$$
where
$f_1$ and $f_2$ are local defining functions of $C_1$ and $C_2$ around the point $p$.
\end{definition}

Definition \ref{definition-local} extends to $\mathbb{R}$-divisors by linearity. For instance, say we have a curve $C$ and a $\mathbb{R}$-divisor $\O$ meeting at the point $p$. We decompose $\O$ as $\O=\sum_ia_iZ_i$, where $Z_i$'s are distinct prime divisors
and $a_i\in\mathbb{R}$. Then,
$$
(C.\O)_p:=\sum_ia_i(C. Z_i)_p,$$
where $(C. Z_i)_p=0$ if $Z_i$ does not pass through the point $p$.
A useful fact we will use often is that under a blow-up
the local intersection number changes as follows,
\beq
\label{multblowup}
(\tilde{C}.\tilde{\Omega})_q
\le (C.\Omega)_p-\mult_p\Omega,
\qq \h{with equality if $C$ is smooth at $p$.}
\eeq

The classical inversion of adjunction on surfaces has the following well-known consequence
\cite{KollarNotes},\cite[Theorem~7]{Cheltsov}.

\begin{lemma}
\lb{ioalem}
Let $C$ be an irreducible curve on a surface $S$, and let $p$ be a smooth point in both $C$ and $S$. Let $a\in\QQ\cap [0,1]$, and let $\Omega$ be an effective $\mathbb{Q}$-divisor
on $S$ with $C\not\subset\supp\O$. If
$$
(C.\Omega)_p\le
1,
$$
then $(S,aC+\Omega)$ is log canonical at $p$.
\end{lemma}

Lemma \ref{ioalem} can be improved by taking into account the parameter $a$ as well as the vanishing order of $\O$.
Throughout this section we set
$$
m:=\mult_p\Omega.
$$

\begin{prop}
\lb{ioa2prop}
Let $C$ be an irreducible curve on a surface $S$, and let $p$ be a smooth point in both $C$ and $S$. Let $a\in\QQ\cap [0,1]$, and let $\Omega$ be an effective $\mathbb{Q}$-divisor
on $S$ with $C\not\subset\supp\O$. Suppose that
$$
(C.\Omega)_p\le
\begin{cases}
2-a, &  \h{\rm if \ }m\le1,\cr
1, & \h{\rm if \ } m> 1.\cr
\end{cases}
$$
Then $(S,aC+\Omega)$ is log canonical at $p$.
\end{prop}

\begin{proof}
The case $m>1$ follows from Lemma \ref{ioalem}
(actually, regardless of $m$).

Suppose $m\le1$.
Let $\pi:\tilde{S}\rightarrow S$ be the blow-up at the point $p$,
with exceptional curve $\pi^{-1}(p)=E$, and let $\tilde{C}$ and $\tilde{\Omega}$ denote the proper transforms of $C$ and $\O$.
Then the log pair $(S,aC+\Omega)$ lifts to $(\tilde{S},a\tilde{C}+\tilde{\Omega}+(a+m-1)E)$.
Since $a,m\le 1$ by assumption, $a+m-1\le 1$ so the latter pair is lc at a general point of $E$.
It remains to check lc at the intersection points of $E$ with $\tilde\O$ and $\tilde C$.
First, let $q\in (E\cap\tilde\O)\sm \tilde C$. 
Then,
$(E.\tilde{\Omega})_{q} 
\le E.\tilde{\Omega} = m \le 1$, so by Lemma \ref{ioalem}  our log pair
is lc at $q$.
Second, let $\{q\}= E\cap\tilde{C}$.
Again, by  Lemma \ref{ioalem}, it suffices to check that
$\big(\tilde{C}.(\tilde{\Omega}+(a+m-1)E)\big)_{q}\le1$,
and since $(\tilde{C}.\tilde{\Omega})_q=(C.\Omega)_p-m$
(by \eqref{multblowup})
and $(\tilde{C}. E)_q=1$ this amounts to
$(C.\Omega)_p+a-1\le 1,$
precisely our assumption.
Thus, $(\tilde{S},a\tilde{C}+\tilde{\Omega}+(a+m-1)E)$ is lc along $E$, equivalently
$(S,aC+\Omega)$ is lc at $p$.
\end{proof}

We continue with a new local inequality incorporating also an additional
``boundary curve".

\begin{theorem}

\label{theorem:local-inequality}

Let $B$ and $C$ be irreducible curves on a surface $S$
intersect transversally at a  point $p$ that is smooth in
$B,C$ and $S$. Let $a,b\in\QQ\cap [0,1)$, and let $\Omega$ be an effective $\mathbb{Q}$-divisor
on $S$ with $B,C\not\subset\supp\O$.
Suppose that
$$
(B.\Omega)_p\le
\begin{cases}\dis
\frac{m}{(m-b)_+}(1-a)-b & \h{\rm if \ }
m\in(0,1] \h{\rm \ and either
$a+(C.\Omega)_p-b\le 1$ or $a+m\le 1$},\cr
1-a & \h{\rm if \ }
m> 1.\cr
\end{cases}
$$
Then $(S,(1-b)B+aC+\Omega)$ is log canonical at $p$.

\end{theorem}

\def\tB{\tilde B}
\def\tC{\tilde C}
\def\tO{\tilde \O}
\def\tp{\tilde p}

Note that the case $m = 0 $  is trivial.
Here, $(x)_+:=\max\{x,0\}$. Thus, when $m\le b$
and either
$a+(C.\Omega)_p-b\le 1$ or $a+m\le 1$,
we are not assuming anything on $(B.\Omega)_p$.

\begin{proof}
The case $m>1$ follows from Lemma \ref{ioalem}
(actually, regardless of $m$), since
$(B.(aC+\O))_p\le 1$ if and only if $(B.\Omega)_p\le 1-a$
as $(B.C)_p=1$ as they intersect transversally  at $p$.
The case $b=0$ also follows from Lemma \ref{ioalem}.

Suppose then $m\le1$ and $b>0$.
We will use an inductive argument.
Let $\pi:S_2\rightarrow S$ be the blow-up at the point $p$,
with exceptional curve $\pi^{-1}(p)=E_1$, and let $B_2,C_2,\Omega_2$ denote
the corresponding proper transforms of $B,C,\O$.
Then the log pair $(S,(1-b)B+aC+\Omega)$
lifts to $({S_2},(1-b) B_2+a{C}_2+{\Omega}_2+(a+m-b)E_1)$.

Let 
$$
\{p_2\}:=B_2\cap E_1, \qq 
\{q_C\}:=C_2\cap E_1.
$$
First, let $q\in (E_1\cap\O_2)\sm\{q_C,p_2\}$. Then,
$(E_1.{\Omega}_2)_{q}\le m \le 1$, so by Lemma \ref{ioalem}  our log pair
is lc at $q$.

Second, let us consider $q_C$. 
Lemma \ref{ioalem} applied to $C_2$ and $E_1$ yield
lc at $q_C$ if either
$$
1\ge \big(C_2.((1-b) B_2+{\Omega}_2+(a+m-b)E_1)\big)_{q_C}=
(C_2.{\Omega}_2)_{q_C}+a+m-b,
$$
(note $C_2$ and $B_2$ do not intersect at ${q_C}$)
or
$$
1\ge \big(E_1.((1-b) B_2+{\Omega}_2\big)_{q_C}
=
a+(E_1.{\Omega}_2)_{q_C}
$$
(note $E_1$ and $B_2$ do not intersect at $q_C$).
Since $(C_2.{\Omega}_2)_{q_C}= (C.{\Omega})_p-m$ 
by \eqref{multblowup} the first inequality
holds if $a+(C.\Omega)_p-b\le 1$. Since
$(E_1.{\Omega}_2)_{q_C}\le m$ the second inequality
holds if $a+m\le 1$. Thus, by our assumptions, one of these must hold,
so our pair is lc at $q_C$.

It remains to consider $p_2$. 
Lemma \ref{ioalem} yields lc at $p_2$ if
$$
1\ge \big(E_1.((1-b) B_2+a{C}_2+{\Omega}_2)\big)_{p_2}
=1-b+(E_1.{\Omega}_2)_{p_2},
$$
i.e.,  if $(E_1.{\Omega}_2)_{p_2}\le b$, so
in particular if $(E_1.{\Omega}_2)_{p_2}\le m\le b$. We
are therefore done, unless
\begin{equation}
\begin{aligned}
\label{mbeq}
m>b,
\end{aligned}
\end{equation}
which we henceforth assume.

Set
$$m_2:=\mult_{p_2}\Omega_2.$$
At this point, we can already set up an inductive argument to conclude the proof; instead, for the sake of clarity, let us carry through most of the first step in induction before switching to the general step.
To start the induction, let us verify that the pair
$$
({S_2},(1-b) B_2+a{C}_2+{\Omega}_2+(a+m-b)E_1),
$$
or, equivalently (as we are working at $p_2$, away from $C_2$),
$$
({S_2},(1-b) B_2+{\Omega}_2+(a+m-b)E_1)
$$
satisfies the assumptions of the Theorem. So $E_1$ is our new ``$C$'',
$B_2$ is our new ``$B$'', $\O_2$ is our new ``$\O$'', and the
new ``$a$'' is
$$
a_2:=a+m-b>a.
$$
First, note that $a_2\in[0,1)$, actually even $a_2\le 1-b$ (recall $b>0$ throughout): if
$a+m\le 1$ this is obvious, and if $a+(C.\Omega)_p-b\le 1$
then as $m\le (C.\Omega)_p$ we are also done.
Note also that
$m_2\le m\le 1$ since multiplicities cannot increase
under blow-ups.
So, it remains to check that
\begin{equation}
\begin{aligned}
\label{either2ineq}
a_2+(E_1.\Omega_2)_{p_2}-b\le 1 \h{\rm \ \ \ or \ \  \ } a_2+m_2\le 1,
\end{aligned}
\end{equation}
and that
\begin{equation}
\begin{aligned}
\label{step2ineq}
(B_2.\Omega_2)_{p_2}\le
\frac{m_2}{m_2-b}(1-a_2)-b.
\end{aligned}
\end{equation}
Let us first check \eqref{either2ineq}. The key is to use the estimate on $(B.\Omega)_p$ in the statement,
as we will see shortly.
In fact, we will prove that a statement stronger than \eqref{either2ineq} holds:
\begin{equation}
\begin{aligned}
\label{either3ineq}
a_2+m-b\le 1 \h{\rm \ \ \ or \ \  \ } a_2-m+(B.\O)_p\le 1.
\end{aligned}
\end{equation}
The first inequality is stronger since $(E_1.{\Omega}_2)_{q}\le m$,
while the second inequality is stronger since $(B.\O)_p=(B_2.\O_2)_{p_2}+m\ge m_2+m$.
Now, the first inequality in \eqref{either3ineq} can be written as
$a+m-2b+m\le 1$ or $2\le \frac{1-a}{m-b}$, while the second inequality
can be written as
$1-a+b\ge m+m_2\ge (B.\O)_p$ or
$2\ge \frac{(B.\O)_p-(1-a)+b}{b}$. By our assumption
$
 \frac{m}{m-b}(1-a)-b
 \ge
 (B.\Omega)_p
$,
so we can write yet stronger inequalities:
\begin{equation}
\begin{aligned}
\label{either4ineq}
2\le \frac{1-a}{m-b} \h{\rm \ \ \ or \ \  \ }
2\ge \frac{\frac{m}{m-b}(1-a)-b-(1-a)+b}{b}=\frac{1-a}{m-b},
\end{aligned}
\end{equation}
which it trivially true, concluding the proof of \eqref{either2ineq}.
It remains to check 
\eqref{step2ineq}.
Let us prove this as part of an inductive argument
(that will yield  \eqref{step2ineq} as 
the first step in the induction, i.e., by setting $k=2$ below).

Then we are in exactly the same setting as before, and we may blow-up $S_2$
at $p_2$ as all the conditions of Theorem \ref{theorem:local-inequality}
are satisfied for
$({S}_2,(1-b)B_2+{\Omega}_2+(a+m-b)E_1)$ at $p_2=B_2\cap E_1$.
Let $k\in\NN$. By induction on the number of blow-ups, assume that we have
performed $k-1$ blow-ups at $p_1:=p,p_2,\ldots,p_{k-1}$ with exceptional
divisors $E_1,\ldots,E_{k-1}$, where $B_i$ is the proper transform
of $B_{i-1}$ (and $B_1:=B$, $\O_1:=\O, C_1:=C$) and where $p_k=E_{k-1}\cap B_k$,
and that at for each $i=0,\ldots,k-1$ we have (set $m_1:=m$)
\begin{equation}
\begin{aligned}
\label{miEq}
m_{i+1}>b,
\end{aligned}
\end{equation}
and (set $a_1=a$)
\begin{equation}
\begin{aligned}
\label{aiEq}
a_{i+1}=a_i+m_i-b,
\end{aligned}
\end{equation}
and
$$
(B_i.\Omega_i)_{p_i}\le
\frac{m_i}{m_i-b}(1-a_i)-b,
$$
and  that
$$
\h{\rm\ \ \  either $a_i+(C_i.\Omega_i)_{p_i}-b\le 1$ or $a_i+m_i\le 1$}.
$$
We need to show that
$$
(B_k.\Omega_k)_{p_k}\le
\frac{m_k}{m_k-b}(1-a_k)-b
\h{\rm\ \ \  if either $a_k+(C_k.\Omega_k)_{p_k}-b\le 1$ or $a_k+m_k\le 1$}
$$
First, note that indeed either $a_k+(C_k.\Omega_k)_{p_k}-b\le 1$ or $a_k+m_k\le 1$:
this is checked in each step just as we did for $k=2$ with the number ``2'' in
  \eqref{either4ineq} being replaced by $k$ (we omit the details).
Note that $(B_k.\Omega_k)_{p_k}=
(B_{k-1}.\Omega_{k-1})_{p_{k-1}}-m_{k-1}$, so it suffices to show that
$$
(B_{k-1}.\Omega_{k-1})_{p_{k-1}}-m_{k-1}
\le \frac{m_k}{m_k-b}(1-a_k)-b.
$$
Since $m_k\le m_{k-1}$ and since
\begin{equation}
\begin{aligned}
\label{xxmbeq}
\frac x{x-b} \h{\ \ is decreasing in\ \ } x\in(b,\infty),
\end{aligned}
\end{equation}
 it suffices to show (recall \eqref{miEq}) that
$
(B_{k-1}.\Omega_{k-1})_{p_{k-1}}-m_{k-1}
\le \frac{m_{k-1}}{m_{k-1}-b}(1-a_k)-b,
$
 i.e.,
$$
(B_{k-1}.\Omega_{k-1})_{p_{k-1}}-m_{k-1}
\le \frac{m_{k-1}}{m_{k-1}-b}(1-a_k+m_{k-1}-b)-b
=
\frac{m_{k-1}}{m_{k-1}-b}(1-a_{k-1}-b)-b,
$$
so we are done, by induction, if we can show that an infinite number of blow-ups is impossible.
This is indeed so, since
$
(B_k.\Omega_k)_{p_k}=(B_{k-1}.\Omega_{k-1})_{p_{k-1}}-m_{k-1},
$
and by induction $m_{i+1}>b$, so after at most
$N\!:=\lceil (B.\Omega)_{p}/b\rceil$ blow-ups we would have $m_{N+1}<b$
and then the pair would be lc at $p_{N+1}$ by our original argument (just before  \eqref{mbeq}) using Lemma
\ref{ioalem}  with no further
need to blow-up.
\end{proof}

Most of the technical assumptions in
Theorem \ref{theorem:local-inequality}
can actually be removed, to yield the following
elegant and very useful criterion.

\begin{cor}
\label{theorem:local-inequality-II}

Let $B$ and $C$ be irreducible curves on a surface $S$
intersect transversally at a  point $p$ that is smooth in
$B,C$ and $S$. Let $a,b\in\QQ\cap [0,1]$, and let $\Omega$ be an effective $\mathbb{Q}$-divisor
on $S$ with $B,C\not\subset\supp\O$.
Suppose that
$$
(B.\Omega)_p\le
\begin{cases}\dis
\frac{(C.\Omega)_p}{((C.\Omega)_p-b)_+}(1-a)-b & \h{\rm if \ }
m\in(0,1],\cr
1-a & \h{\rm if \ }
m> 1.\cr
\end{cases}
$$
Then $(S,(1-b)B+aC+\Omega)$ is log canonical at $p$.

\end{cor}

Note that the case $m=0$, i.e., $(C.\Omega)_p = 0$ is trivial since then also $(B.\Omega)_p = 0$
and $(S,(1-b)B+aC+\Omega)$ is lc at $p$ iff $(S,(1-b)B+aC)$ is which is true
as $1-b,a\le 1$.

\begin{proof}
If $b=0$ or if $m>1$ then the assumption is $(B.\Omega)_p\le 1-a $ so we are done by
Proposition  \ref{ioa2prop}.
Suppose now that $b>0$ and $m\le1$.

First, $m\le (C.\Omega)_p$.
Thus, by \eqref{xxmbeq},
$$
(B.\Omega)_p\le
\dis
\frac{m}{(m-b)_+}(1-a)-b.
$$
Thus, if
\begin{equation}
\begin{aligned}
\label{eitheroreq}
\h{either $a+(C.\Omega)_p-b\le 1$ or $a+m\le 1$},
\end{aligned}
\end{equation}
Theorem \ref{theorem:local-inequality}
is applicable (the cases $b=1$ or $a=1$ are handled separately
by Proposition  \ref{ioa2prop} since for either one we may take
$B=0$)
 and we are done.
We claim that
 always holds. First, if
$(C.\Omega)_p\le b$ the first inequality in  \eqref{eitheroreq}
automatically holds. So, suppose $((C.\Omega)_p-b)_+=(C.\Omega)_p-b>0$.
Since $(B.\Omega)_p\ge  m$,
\begin{equation*}
\begin{aligned}
\label{}
m
&\le
(B.\Omega)_p
\le
\frac{(C.\Omega)_p}{(C.\Omega)_p-b}(1-a)-b
=1-a-b+\frac {b(1-a)}{(C.\Omega)_p-b},
\end{aligned}
\end{equation*}
i.e.,
$a+m
\le1-b\frac {a+(C.\Omega)_p-b-1}{(C.\Omega)_p-b},
$
implying \eqref{eitheroreq}.
\end{proof}

\section{Vanishing order estimates}
\label{section:many-blow-ups}

For the remainder of the article we will concentrate on the proof of Theorem \ref{main}. Set,
$$
\overline{S}:=\mathbb{P}^1\times\mathbb{P}^1, \qq
\overline{C}:=\h{a smooth curve of bi-degree $(1,2)$}\subset\overline{S}.
$$
Denote by $\overline{F}$ a general line of bi-degree $(1,0)$
 and by $\overline{G}$ a general line of bi-degree $(0,1)$.

Note that, the curve $\overline{C}$ is, by definition, cut out by a bi-degree $(1,2)$ polynomial. To be more precise, let $([s:t],[x:y])$ be the bi-homogeneous coordinate system on $\overline{S}$. Then $\overline{C}$ is cut out by some polynomial $F(s,t,x,y)$ that is homogeneous with degree 1 in $s,t$ variables and homogeneous with degree 2 in $x,y$ variables. Up to a coordinate change, we may assume $F(s,t,x,y)=sy^2-tx^2$, for simplicity. (Indeed, we may assume $F=sP(x,y)+tQ(x,y)$. If both $P$ and $Q$ are  of a linear polynomial squared, we are done. Assume that at least one of them is not a square. Apply a coordinate change to $x,y$ so that  $F=Csxy+tQ(x,y)$ with $Q(x,y)=x^2+axy+y^2$, and let $x\rightarrow x+y$, $y\rightarrow x-y$, so in
the new coordinates $F=Cs(x^2-y^2)+t((2+a)x^2+(2-a)y^2)$. Finally, 
apply a linear coordinate change to $s,t$.)

The linear system $|\overline{F}|$ contains exactly two curves
that are tangent to $\overline{C}$.
Denote them by
$$
\overline{F}_0,\overline{F}_\infty,
$$
and let
$$
\overline{p_0}:=\overline{F}_0\cap \overline{C},\qquad
\overline{p_\infty}:=\overline{F}_\infty\cap \overline{C}.\qquad
$$
In $([s:t],[x:y])$ coordinates, one simply has $\overline{F}_0=\{s=0\},\ \overline{F}_\infty=\{t=0\}$ and $p_0=([0:1],[0;1])$ and $p_\infty=([1:0],[1:0])$.
Let $\overline{F}_1,\ldots,\overline{F}_r$ be
distinct  bi-degree $(1,0)$ curves in $\overline{S}$
that are all different from the curves $\overline{F}_0$ and $\overline{F}_\infty$.
Then each intersection $\overline{F}_i\cap\overline{C}$ consists of two points.
For each $i=1,\ldots r$, let
$$
\overline{p}_i\in \overline{F}_i\cap \overline{C}
$$
be one of these two points.

Let
$$
I:=\{i_1,\ldots,i_r\}\subset \{0,1,...,r,\infty\},
$$
let $\pi\colon S\to\overline{S}$ be the blow-up of $\overline{S}$ at the $r$ points $\{\overline p_i\}_{i\in I}\subset\{\overline p_0,\overline p_1,...,\overline p_r,\overline p_\infty\}$,
and denote by 
$$
E_{j}:=\pi^{-1}(\overline p_{j}), 
\qq
j\in I, 
$$ 
the exceptional curves of $\pi$. To be precise,
we note that 
we are blowing-up $r$ of the $r+2$ points 
$\{\overline p_0,\overline p_1,...,\overline p_r,\overline p_\infty\}$.
Denote by
$$
F_0,F_1,\ldots,F_r,F_{\infty}
$$
the proper transform on the surface $S$ of the curves $\overline{F}_0,\overline{F}_1,\ldots,\overline{F}_r,\overline{F}_{\infty}$
(note that exactly $r$ of these are $-1$-curves and the remaining
two are are $0$-curves).
Let $C$ be the proper transform of the curve $\overline{C}$,
so
$$
C=\pis\overline C
-\sum_{j\in I}E_j
\sim\pis(\overline F+2 \overline G)-\sum_{j\in I}E_j.
$$
Let
$$
K_\be:=K_S+(1-\beta)C.
$$
Then, as $-K_{\overline{S}}=2\overline{F}+2\overline{G}$, and
$K_S=\pis K_{\overline{S}} + \sum_{j\in I} E_j$, 
\begin{equation}
\begin{aligned}
\label{Kbetaeq}
-K_\be\sim_{\mathbb{Q}}2\pis\overline F+2\pis\overline G-\sum_{j\in I} E_j
-(1-\be)(\pis \overline C-\sum_{j\in I}E_j)
\sim_{\mathbb{Q}} \pis\overline F+\beta C,
\end{aligned}
\end{equation}
thus $(S,C)$ is asymptotically log Fano,
more precisely \cite[(4.2)]{CR},
\begin{equation}
\begin{aligned}
\label{Kbetaample}
\h{ $-K_\be$ is ample for $0<\beta<\frac{2}{r-4}$}.
\end{aligned}
\end{equation}

Below, we will always assume
\begin{equation}
\begin{aligned}
\label{r7}
r\ge7.
\end{aligned}
\end{equation}
To prove Theorem~\ref{main}, we will show in \S\ref{blctSec}
that for some $\la>1$ (independent of $k$)
and for any $k$-basis divisor $D$, the log pair
\begin{equation}
\label{equation:many-blow-ups-log-pair}
\big(S,(1-\beta)C+\la D\big)
\end{equation}
has lc singularities for sufficiently small $\beta>0$,
and sufficiently large $k$.
To do this, in the present section we obtain
explicit estimates on the
order of vanishing of basis divisors.
This involves estimating integrals appearing in
Corollary \ref{FujCor}, which in turn involves elementary computations of Seshadri constants
and pseudoeffective thresholds.

For the estimate on the order of vanishing, we
require the Seshadri constant and the pseudoeffective
threshold (recall \eqref{tauEq}). Let us recall the definition of the former.
The Seshadri constant of
$(X,Z)$ with respect to $L$,
\beq\label{espilonXZLEq}
\sigma(Z,L)=\sup\big\{c>0\,:\, L-cZ \h{\rm \ is ample}\big\}.%
\eeq

We start by computing $\tau(-K_\be,Z)$.

\begin{lemma}
\label{taulemma}
Let $0<\beta<\frac{2}{r-4}$.
One has,

\begin{equation}
\begin{aligned}
\label{}
\tau(-K_\be,Z)=
\begin{cases}
 1 & \h{if $Z$ be an irreducible curve in $|\pis \overline{F}|$,}\cr
\be& \h{if $Z=C$,} \cr
 1 & \h{if $Z\in \cup_{i\in I}\{E_i,F_i\}$}. 
\end{cases}
\end{aligned}
\end{equation}
\end{lemma}

\begin{proof}
If $Z$ is an irreducible curve in $|\pis \overline{F}|$, as
$-K_\be \sim_{\mathbb{Q}} \pis \overline{F}+\beta C$
and $\pis \overline{F}$ is effective (has 0 self-intersection)
and $C$ has zero volume (as $C^2=4-r<0$), we must have
$\tau(-K_\be,Z)= 1$.

If $Z=C$,
$-K_\be-xC \sim_{\mathbb{Q}} \pis \overline{F}+(\beta-x) C$. For $x=\be$ we get
$\vol(-K_\be-\be C)=\vol(\pis \overline{F})=(\pis \overline{F})^2=0$, so
 $\tau(-K_\be,C)= \be$.

If $Z=E_1$, say,
we claim
$\vol(-K_\be-E_1)=0$,
i.e., $\tau(-K_\be,E_1)\le 1$.
Since
$-K_\be-xE_1 \sim_{\mathbb{Q}} \pis \overline{F}+\beta C-xE_1 \sim_{\mathbb{Q}} F_1+(1-x)E_1+\be C$ is effective for $x\in[0,1]$, we would thus have
$\tau(-K_\be,E_1)=1$.
To prove the claim,
Lemma \ref{lemma:vol-replacement-of-line-bundle} implies that
$\vol(-K_\be-E_1)=\vol(-K_\be-E_1-(1-\be)F_1)=\vol(\be(F_1+C))=\be^2\vol(F_1+C)=0$ by Claim \ref{volzeroclaim} below.
If $Z=F_1$, say, the computations are similar.
By Remark \ref{Iremark}, we are done.
\end{proof}

\begin{claim}
\lb{volzeroclaim}
Let $i\in I$. Then, $\vol(F_i+C)=0$.
\end{claim}

\bpf 
When $r\le 5$ the divisor $F_i+C$ is nef and has nonpositive
self-intersection, so the claim follows. Suppose $r>5$. 
Repeated application of Lemma \ref{lemma:vol-replacement-of-line-bundle} implies that
$\vol(F_i+C)=\vol(F_i+C-\frac{r-5}{r-4}C)=
\vol\big(F_i+C-\frac{r-5}{r-4}C-(1-\frac1{r-4})F_i\big)=
\ldots = \vol(a_jF_i+b_jC)$.
The sequences $\{a_j,b_j\}$ are decreasing so let $a:=\lim a_j$, $b:=\lim b_j$. This process will stop if $aF_i+bC$ is nef.
But that implies (by intersecting with $F_i$ and $C$)
that $a,b\ge0$ and $-a+b\ge0$ and $a+(4-r)b\ge0$; adding up
and using that $r>5$ we see that $b=0$ which then implies $a=0$,
so the claim follows.
\epf 

\bremark
\lb{Iremark}
As we just saw above, the computations depend only on the
intersection-theoretic properties of $Z$, so they are exactly the same
if $Z$ is any element of $\cup_{i\in I}\{E_i,F_i\}$
(in particular also if $\{0,\infty\}\cap I\not=\emptyset$).
\eremark

By \eqref{Kbetaample}, $-K_\be$ is ample for small $\be$, so it makes sense  $\s(-K_\be,Z)$.

\begin{lemma}
\label{sigmalemma}
Let $0<\beta<\frac{2}{r-4}$.
One has,
\begin{equation}
\begin{aligned}
\label{}
\sigma(-K_\be,Z)=
\begin{cases}
 1-\beta(r-4)/{2} & \h{if $Z$ be an irreducible curve in $|\pis \overline{F}|$,}\cr
\be& \h{if $Z=C$,} \cr
 \be & \h{if $Z\in\cup_{i\in I}\{E_i,F_i\}
 $.} \cr
\end{cases}
\end{aligned}
\end{equation}
\end{lemma}

\begin{proof}
In the first case, using \eqref{Kbetaeq},
$$
(-K_\be -x Z). C=(1-x)\pis \overline F. C+\beta C^2=2(1-x)-\beta(r-4),
$$
(while $(-K_\be -x Z). \pis \overline F=2\be>0$), i.e.,
$
\sigma(-K_\be,Z)=1-\beta(r-4)/{2}.
$
In the second case,
$$
(-K_\be-xC).C
=(\pis \overline{F}+(\beta-x) C).C
=2-(x-\be)(r-4),
$$
while
$$
(-K_\be-xC).\pis \overline{F}
=(\pis \overline{F}+(\beta-x) C).\pis \overline{F}
=2(\be-x),
$$
so  $
\sigma(-K_\be,C)=\beta.
$
In the third case, say $1\in I$,
$$
(-K_\be -x E_1). C=
(F_1+(1-x)E_1+\be C).C
=1+1-x+\be(4-r)
,
$$
while
$$
(-K_\be -x E_1). F_1=
(F_1+(1-x)E_1+\be C).F_1
=-1+1-x+\be=\be-x
,
$$
and
$$
(-K_\be -x E_1). E_1=
(F_1+(1-x)E_1+\be C).E_1
=1-(1-x)+\be=\be+x
,
$$
so $
\sigma(-K_\be,E_1)=\beta.
$
If $Z=F_1$, say, the computations are similar.
By Remark \ref{Iremark}, we are done.
\end{proof}

\begin{lemma}
\label{ordlemma}
Let $0<\beta<\frac{2}{r-4}$ and
let $D$ be a $k$-basis divisor.
One has,
\begin{equation}
\begin{aligned}
\label{}
\ord_ZD\le
\begin{cases}
\dis \frac{1}{2}-\frac{\beta(r-4)}{8}+O(\beta^2)+\epsilon_k, & \h{if $Z$ be an irreducible curve in $|\pis \overline{F}|$,}\cr \cr
\dis \frac{\beta}{2}+O(\beta^2)+\epsilon_k
,& \h{if $Z=C$,} \cr\cr
\dis  \frac{1}{2}-\frac{\beta(r-6)}{8}+O(\beta^2)+\epsilon_k,  & \h{if $Z\in \cup_{i\in I}\{E_i,F_i\} 
$,} \cr
\end{cases}
\end{aligned}
\end{equation}
with
$\lim_k\eps_k=0$.

\end{lemma}

\begin{proof}
The result follows from Lemmas \ref{taulemma} and \ref{sigmalemma} Corollary \ref{FujCor} by estimating the volume integral
\begin{equation}
\begin{aligned}
\label{tauintesteq}
\frac{1}{K_\be^2}\int_0^\tau \vol (-K_\be-xZ)dx
&=
\frac{1}{K_\be^2}\int_0^\sigma (K_\be+xZ)^2dx
+
\frac{1}{K_\be^2}\int_\sigma^\tau \vol (-K_\be-xZ)dx
\cr
&=
\frac{1}{K_\be^2}\int_0^\sigma (K_\be^2+2xZ.K_\be+x^2Z^2)dx
+
\frac{1}{K_\be^2}\int_\sigma^\tau \vol (-K_\be-xZ)dx
\cr&=
\sigma+
\frac{1}{K_\be^2}
\Big[
\frac{Z^2}{3}\sigma^3
+
{Z.K_\be}\sigma^2
+
\int_\sigma^\tau \vol (-K_\be-xZ)dx
\Big]
.
\end{aligned}
\end{equation}
We will also use the estimate
\begin{equation}
\begin{aligned}
\label{sigmataurougheq}
\int_\sigma^\tau \vol (-K_\be-xZ)dx\le
(\tau-\sigma)(K_\be+\sigma Z)^2
=
(\tau-\sigma)(K_\be^2+2\s Z.K_\be+\s^2Z^2).
\end{aligned}
\end{equation}

First, let $Z$ be an irreducible curve in $|\pis \overline{F}|$.
Then  $Z^2=0$, and
$$
Z.K_\be=-2\be,\q K_\be^2=4\beta-\beta^2(r-4),
$$
thus
$$
K_\be^2+2\s Z.K_\be+\s^2Z^2
=
4\beta-\beta^2(r-4)-
(2-(r-4)\be)2\be
=O(\be^2),
$$
and as
$\tau-\sigma=O(\be)$, we get
  \eqref{sigmataurougheq}$=O(\be^3)$.
Thus,
\begin{equation*}
\begin{aligned}
\frac{1}{K_\be^2}\int_0^\tau \vol (-K_\be-x\pis \overline{F})dx
&
\le
\sigma+
\frac{-2\be\sigma^2+O(\be^3)}{4\beta-\beta^2(r-4)}
\cr
&
=
1- \frac{r-4}{2}\be
+
\frac{-2\be+O(\be^3)}{4\beta-\beta^2(r-4)}
=
\frac12- \frac{r-4}{2}\be +O(\be^2).
 \end{aligned}
\end{equation*}

Second, let $Z=C$.
Then $\s=\tau=\beta$, i.e.,  \eqref{sigmataurougheq}$=0$.
Thus,
\begin{equation*}
\begin{aligned}
\frac{1}{K_\be^2}\int_0^\tau \vol (-K_\be-xC)dx
&
=
\be+
\frac{\frac{4-r}3\be^3+(-2+\be(r-4))\be^2}{4\beta-\beta^2(r-4)}
=
\frac\be2 +O(\be^2).
 \end{aligned}
\end{equation*}

Third, say $1\in I$ and let $Z=E_1$ (the proof for $Z=F_1$ is identical as $E_1$ and $F_1$
play symmetric roles in the computations). Then, as $\sigma=-E_1.K_\be=\be$,
\begin{equation}
\begin{aligned}
\label{}
\frac{1}{K_\be^2}\int_0^\tau \vol (-K_\be-xZ)dx
&=
\be+
\frac{1}{4\beta-\beta^2(r-4)}
\Big[O(\be^3)
+
\int_\be^1 \vol (-K_\be-xZ)dx
\Big]
\cr
&=
\be+O(\be^2)+
\frac{1}{4\beta-\beta^2(r-4)}
\int_\be^1 \vol (-K_\be-xZ)dx
.
\end{aligned}
\end{equation}
The remaining integral can be simplified using Lemma \ref{lemma:vol-replacement-of-line-bundle}.
Indeed, $(-K_\be-xZ).F_1=\beta-x,$~so
$$
\vol(-K_\be-xZ)=\vol (-K_\be-xZ-(x-\beta)F_1),\qquad x\in(\be,1).
$$
The divisor on the right hand side is nef for
$$
x\le \s':= \sigma(-K_\be+\be F_1,Z+F_1)=1+{\beta}(5-r)/2
$$
since
$
-K_\be-xZ-(x-\beta)F_1\simq
\sim(1-x)E_1+(1-x+\beta)F_1+\be C
$
and this intersects non-negatively with $E_1$ and $F_1$
while intersecting with $C$ gives $2-2x+\be(5-r)$.
Similarly to \eqref{sigmataurougheq}, we estimate
$$
\begin{aligned}
\int_{\s'}^1 \vol (-K_\be-xZ)dx
&=
\int_{\s'}^1 \vol \big(-K_\be+\be F_1-x(Z+F_1)\big)dx
\cr
&\le
\be\frac{r-5}2
\Big((K_\be-\be F_1)^2+2\s'(Z+F_1).(K_\be-\be F_1)
\cr
&\le
\be\frac{r-5}2
\Big(4\be+\be^2(5-r)-4\be\s'
\Big)=O(\be^3),
\end{aligned}
$$
as $(Z+F_1)^2=0$ and $(K_\be-\be F_1)^2=4\be+\be^2(5-r),
(Z+F_1).(K_\be-\be F_1)=-2\be$.
Next, 
$$
\begin{aligned}
\int_\s^{\s'} \vol (-K_\be-xZ)dx
&=
\int_\s^{\s'} \vol \big(-K_\be+\be F_1-x(Z+F_1)\big)dx
\cr
&=
\int_0^{\s'}-\int_0^\s\vol \big(-K_\be+\be F_1-x(Z+F_1)\big)dx
,
\end{aligned}
$$
and we can compute each integral
as in  \eqref{tauintesteq}, namely,
$$
\begin{aligned}
\int_\s^{\s'} \vol (-K_\be-xZ)dx
&=
\s'(K_\be-\be F_1)^2+
\frac{(Z+F_1)^2}{3}(\s')^3
+
{(Z+F_1).(K_\be-\be F_1)}(\s')^2
\cr
&\q
-
\s (K_\be-\be F_1)^2
-
\frac{(Z+F_1)^2}{3}\s^3
-
{(Z+F_1).(K_\be-\be F_1)}\s^2
\cr
&=
\Big(1+\frac{3-r}2\be\Big)\big(4\be+\be^2(5-r)\big)-
2\be((\s')^2-\s^2)
\cr
&=
\Big(1+\frac{3-r}2\be\Big)\big(4\be+\be^2(5-r)\big)-
2\be\big(1+\be(5-r)+O(\be^2)\big)
\cr
&=
2\be+\be^2(5-r+6-2r+2r-10)+O(\be^3)
=
2\be+\be^2(1-r)+O(\be^3).
\end{aligned}
$$
Altogether, we have shown
\begin{equation*}
\begin{aligned}
\frac{1}{K_\be^2}\int_0^\tau \vol (-K_\be-xZ)dx
&=
\be+O(\be^3)+
\frac{1}{4\beta+\beta^2(4-r)}
\Big[
\int_\sigma^\tau \vol (-K_\be-xZ)dx
\Big]
\cr
&\le
\be+O(\be^3)+
\frac{1}{4\beta+\beta^2(4-r)}
\Big[
2\be+\be^2(1-r)+O(\be^3)+O(\be^3)
\Big]
\cr
&=
\frac{1}{4\beta+\beta^2(4-r)}
\Big[
2\be+\be^2(5-r)+O(\be^2)
\Big]
+O(\be^2)
\cr
&\le\frac{1}{2}-\frac{\beta(r-6)}{8}+O(\beta^2)
,
\end{aligned}
\end{equation*}
as desired. By Remark \ref{Iremark}, we are done.
\end{proof}

\section{Basis log canonical thresholds}
\label{blctSec}

We use the same notation as in Section \ref{section:many-blow-ups}.
The purpose of this section is to prove Theorem~\ref{main} by showing
that for all sufficiently large $k$
and  some $\la>1$ (independent of $k$)
and for any $k$-basis divisor $D\simq -K_\be= -K_S-(1-\beta)C$, the log pair
\begin{equation}
\label{equation:many-blow-ups-log-pair}\big(S,(1-\beta)C+\la D\big)
\end{equation}
has log canonical singularities for sufficiently small $\beta>0$
(independent of $k$).

Let us fix such $\be,k,D$, and set
\begin{equation}
\label{lambdaeq}
\lambda:=1+\frac{\beta}{100}.
\end{equation}
 We split the argument into several lemmas.

\begin{claim}
\lb{claim1}
The pair \eqref{equation:many-blow-ups-log-pair}  is lc
at $S\setminus (C\cup_{i\in I} \{E_i,F_i\})$.
\end{claim}
\begin{proof}
Let $p\in S\setminus  (C\cup_{i\in I} \{E_i,F_i\})$.
Thus, $\pi(p)\not\in\{\overline p_i\}_{i\in I}$, so if we let $\ell$
be the $(1,0)$-curve passing through $\pi(p)$ then
$Z:=\pi^{-1}(\ell)\in|\pis\overline F|$ is a smooth irreducible
curve passing through $p$.
As $p\not \in C$, the pair \eqref{equation:many-blow-ups-log-pair}  is lc
at $p$ if and only if the pair
$\big(S,\la D\big)$ is.
Write
$
\lambda D=\la Z\ord_ZD+\Delta.
$
Then
$$
(Z.\Delta)_p
\le
Z.\Delta
=
Z.(\lambda D-\la Z\ord_ZD)
=
Z.(\lambda (Z+\be C) -\la Z\ord_ZD)
=
2\be\la
\le1,
$$
for $\beta$ small, so we are done by Lemma \ref{ioalem}.
\end{proof}

\begin{claim}
The pair \eqref{equation:many-blow-ups-log-pair}  is lc
at  $\cup_{i\in I} \{E_i,F_i\}\setminus C$.
\end{claim}

\begin{proof}
Say $1\in I$ and let $p\in E_1\sm C$.
Again, it suffices to show
the pair
$\big(S,\la D\big)$ is lc at $p$.
Then
\begin{equation*}
\begin{aligned}
\label{}
(E_1.\Delta)_p
\le
E_1.\Delta
&=
E_1.(\lambda D-\la E_1\ord_{E_1}D)
\cr
&=
E_1.(\lambda (E_1+F_1+\be C) -\la E_1\ord_{E_1}D)
=
\la(\be+\ord_{E_1}D)
\le1,
\end{aligned}
\end{equation*}
for $\beta$ small by Lemma \ref{ordlemma},
so we are done by Lemma \ref{ioalem}.
\end{proof}

\begin{claim}
The pair \eqref{equation:many-blow-ups-log-pair}  is lc
at $C\sm\big(\{p_0,p_\infty\}\cup\bigcup_{i\in I}\{E_i,F_i\} \big)$ .
\end{claim}

\begin{proof}
Let $p\in C\sm\big(\{p_0,p_\infty\}\cup\bigcup_{i\in I}\{E_i,F_i\} \big)$.
As in the proof of Claim \ref{claim1}, let
$Z\in|\pis\overline F|$ be a smooth irreducible
curve passing through $p$.
Write
$$
\lambda D+(1-\beta)C
=
\la Z\ord_ZD+(1-\beta+\la\ord_CD)C+\Omega.
$$
As  $Z$ intersects $C$ transversally at $p$ (as $p\in \{p_0,p_\infty\}$),
by Corollary \ref{theorem:local-inequality-II}
 it suffices to show
that
\begin{equation}
\begin{aligned}
\label{multpOEq}
\mult_p\Omega\leq1
\end{aligned}
\end{equation}
 and that
$$
C.\Omega\le
\frac{(Z.\Omega)_p}{((Z.\Omega)_p-\be+\la\ord_CD)_+}(1-\la \ord_ZD)-\be+\la\ord_CD.
$$
For \eqref{multpOEq}, note that by \eqref{ZOInterEq}
$\mult_p\Omega\le
(Z.\Omega)_p = O(\be).
$
Now,
\begin{equation*}
\begin{aligned}
\label{}
C.\Omega
&=
C.
\big(
\lambda D
-
\la Z\ord_ZD-\la C\ord_CD
\big)
\cr
&=
C.
\big(
\lambda (\pis \overline F+\be C)
-
\la Z\ord_ZD-\la C\ord_CD
\big)
\cr
&=
\la(2+\be(4-r)-2\ord_ZD-(4-r)\ord_CD)
\cr
\end{aligned}
\end{equation*}
by Lemma \ref{ordlemma}, while
\begin{equation}
\begin{aligned}
\label{ZOInterEq}
(Z.\Omega)_p
&\le
Z.\Omega=
Z.
\big(
\la D
-
\la  Z\ord_ZD-\la C\ord_CD
\big)
\cr
&=
2\la(\be-\ord_CD).
\cr
\end{aligned}
\end{equation}
As $\la>1$, this is larger than $\be-\la\ord_CD$,
and using \eqref{xxmbeq}, it
suffices to show
$$
\begin{aligned}
2\la\Big(1-\ord_ZD+\frac{r-4}2(\ord_CD-\be)\Big)
&\le
\frac
{2\la(\be-\ord_CD)}
{\be(2\la-1)-\la\ord_CD}
(1-\la \ord_ZD)-\be+\la\ord_CD,
\cr
\end{aligned}
$$
i.e.,
$$
\begin{aligned}
1-\ord_ZD+\frac{r-5}2(\ord_CD-\be)
&\le
\frac
{\be-\ord_CD}
{\be(2\la-1)-\la\ord_CD}
(1-\la \ord_ZD)+\frac{\la-1}{2\la}\be,
\cr
\end{aligned}
$$
i.e.,
$$
\begin{aligned}
\ord_ZD
\Big(
\frac{\la(\be-\ord_CD)}
{\be(2\la-1)-\la\ord_CD}
-1
\Big)
+
1+\frac{r-5}2(\ord_CD-\be)
&\le
\frac
{\be-\ord_CD}
{\be(2\la-1)-\la\ord_CD}
+\frac{\la-1}{2\la}\be,
\cr
\end{aligned}
$$
i.e.,
$$
\begin{aligned}
\ord_ZD
\frac{\be(1-\la)}
{\be(2\la-1)-\la\ord_CD}
+
1+\frac{r-5}2(\ord_CD-\be)
&\le
\frac
{\be-\ord_CD}
{\be(2\la-1)-\la\ord_CD}
+\frac{\la-1}{2\la}\be.
\cr
\end{aligned}
$$
The first term is negative while the last is positive, and
since $\ord_CD-\be<0$ and
$r-5\ge 2$, suffices to show
$$
\begin{aligned}
1&\le
(\be-\ord_CD)
\Big(\frac
{1}
{\la(\be-\ord_CD)+\be(\la-1)}
+1
\Big)
\cr
\end{aligned}
$$
i.e.,
$$
\begin{aligned}
\la(\be-\ord_CD)+\be(\la-1)
&\le
(\be-\ord_CD)
\Big(
{1}
+
{\la(\be-\ord_CD)+\be(\la-1)}
\Big)
\end{aligned}
$$
i.e.,
\begin{equation}
\begin{aligned}
\lb{claimineq1}
(\be-\ord_CD)
\big(
(\la-1)(1-\be)
-
{\la(\be-\ord_CD)}
\big)
+\be(\la-1)
&\le
0
\end{aligned}
\end{equation}
By Lemma \ref{ordlemma} and \eqref{lambdaeq},
for $\be$ sufficiently
small,
$$
\begin{aligned}
&(\be-\ord_CD)
\big(
(\la-1)(1-\be)
-
{\la(\be-\ord_CD)}
\big)
+\be(\la-1)
\cr
&
\q\le
(\be-\ord_CD)
\big(
(\la-1)(1-\be)
-
{\la\be/3}
\big)
+\be(\la-1)
\cr
&
\q\le
-\frac\be4(\be-\ord_CD)
+\be(\la-1)
\le
-\frac{\be^2}{12}
+\be(\la-1)
\le0,
\end{aligned}
$$
proving \eqref{claimineq1}.
\end{proof}

\begin{claim}
The pair \eqref{equation:many-blow-ups-log-pair}  is lc
at $\{C\cap E_1,\ldots,C\cap E_r,C\cap F_1,\ldots,C\cap F_r\}$.
\end{claim}

\begin{proof}
We prove the result for $p=C\cap E_1$ (the proof for other cases is similar). Write
$$\lambda D+(1-\beta)C=\lambda E_1\ord_{E_1}D+\lambda F_1\ord_{F_1}D+(1-\beta+\lambda \ord_CD)C+\Omega.$$
Note that $E_1$ intersects $C$ transversally at $p$ and $F_1$ does not pass through $p$.
We have
\begin{equation*}
\begin{cases}
\lambda\beta-\lambda\ord_CD+\lambda\ord_{E_1}D-\lambda\ord_{F_1}D=E_1. \Omega\ge  \text{mult}_p\Omega,\\
\lambda\beta-\lambda\ord_CD-\lambda\ord_{E_1}D+\lambda\ord_{F_1}D=F_1. \Omega\ge 0.\\
\end{cases}
\end{equation*}
From these two inequalities we get $\text{mult}_p\Omega\le 2\lambda(\beta-\ord_CD)=O(\be)$ and
\begin{equation}
\label{aE-aF<beta-epsilon}
\la\ord_{E_1}D-\la\ord_{F_1}D\le  \la(\beta-\ord_CD).
\end{equation}
In particular, $\text{mult}_p\Omega\le1$. 
Then by Corollary \ref{theorem:local-inequality-II}
if suffices to show
$$
C.\Omega\le
\frac{(E_1.\Omega)_p}{((E_1.\Omega)_p-\be+\la\ord_CD)_+}(1-\la \ord_{E_1}D)-\be+\la\ord_CD.
$$
Now, using (\ref{aE-aF<beta-epsilon}),
\begin{equation*}
\begin{aligned}
\label{}
C.\Omega
&=
C.
\big(
\lambda D
-
\la E_1\ord_{E_1}D-\la F_1\ord_{F_1}D-\la C\ord_CD
\big)
\cr
&=
C.
\big(
\lambda (\pis \overline F+\be C)
-
\la E_1\ord_{E_1}D-\la F_1\ord_{F_1}D-\la C\ord_CD
\big)
\cr
&=
\la(2+\be(4-r)-\ord_{E_1}D-\ord_{F_1}D-(4-r)\ord_CD)
\cr
&\le
\la(2+\be(5-r)-2\ord_{E_1}D-(5-r)\ord_CD)
\end{aligned}
\end{equation*}
and
\begin{equation*}
\begin{aligned}
\label{}
(E_1.\Omega)_p
&\le
E_1.\Omega=
E_1.
\big(
\la D
-
\la  E_1\ord_{E_1}D-\la F_1\ord_{F_1}D-\la C\ord_CD
\big)
\cr
&=
\la\beta-\la\ord_CD+\la\ord_{E_1}D-\la\ord_{F_1}D
\le2\la(\be-\ord_CD).
\cr
\end{aligned}
\end{equation*}
As $\la>1$, this is larger than $\be-\la\ord_CD$,
and using \eqref{xxmbeq}, it
suffices to show
$$
\begin{aligned}
2\la\Big(1-\ord_{E_1}D+\frac{r-5}{2}(\ord_CD-\be)\Big)
&\le
\frac
{2\la(\be-\ord_CD)}
{\be(2\la-1)-\la\ord_CD}
(1-\la \ord_{E_1}D)-\be+\la\ord_CD,
\cr
\end{aligned}
$$
i.e.,
$$
\begin{aligned}
1-\ord_{E_1}D+\frac{r-6}{2}(\ord_CD-\be)
&\le
\frac
{\be-\ord_CD}
{\be(2\la-1)-\la\ord_CD}
(1-\la \ord_{E_1}D)+\frac{\la-1}{2\la}\be,
\cr
\end{aligned}
$$
i.e.,
$$
\begin{aligned}
\ord_{E_1}D
\Big(
\frac{\la(\be-\ord_CD)}
{\be(2\la-1)-\la\ord_CD}
-1
\Big)
+
1+\frac{r-6}2(\ord_CD-\be)
&\le
\frac
{\be-\ord_CD}
{\be(2\la-1)-\la\ord_CD}
+\frac{\la-1}{2\la}\be,
\cr
\end{aligned}
$$
i.e.
$$
\begin{aligned}
\ord_{E_1}D
\frac{\be(1-\la)}
{\be(2\la-1)-\la\ord_CD}
+
1+\frac{r-6}2(\ord_CD-\be)
&\le
\frac
{\be-\ord_CD}
{\be(2\la-1)-\la\ord_CD}
+\frac{\la-1}{2\la}\be.
\cr
\end{aligned}
$$
The first term is negative while the last is positive, so it suffices to show
$$1-\frac
{\be-\ord_CD}
{\be(2\la-1)-\la\ord_CD}
\le\frac{r-6}{2}(\beta-\ord_CD),$$
i.e.,
$$(\la-1)\frac{2\beta-\ord_CD}{\be(2\la-1)-\la\ord_CD}\le\frac{r-6}{2}(\beta-\ord_CD).$$
Using Lemma \ref{ordlemma}, (\ref{lambdaeq}) and 
\eqref{r7} 
it suffices to show (for $\beta$ small)
$$\frac{\beta}{100}\cdot\frac{2\beta}{\beta-\frac{2}{3}\beta}\leq\frac{1}{2}\cdot(\beta-\frac{2}{3}\beta),$$
i.e., ${3\beta}/{50}\le{\beta}/{6},$
so we are done.
\end{proof}

To finish the proof of Theorem \ref{main}, it remains to
show that the pair \eqref{equation:many-blow-ups-log-pair} is lc at $p=C\cap F_0$ (for $C\cap F_\infty$ the proof is similar).
Then there are two cases to consider, since the argument depends on whether the point $\overline{p_0}$ is blown up or not. The more difficult case is the following:

\begin{claim}
Suppose that $\overline{p_0}$ is not blown up, then the pair \eqref{equation:many-blow-ups-log-pair}  is lc
at $p=C\cap F_0$.
\end{claim}

\begin{proof}
In this case, the curve $F_0$ intersects $C$ tangentially at $p$.
Write
$$\lambda D=\la F_0\ord_{F_0}D+\la C\ord_CD+\Omega.$$
We put
$$m=\mult_p\Omega.$$
Since $2\beta\la=\lambda D.F_0=2\la\ord_CD+\Omega.F_0\ge2\la\ord_CD+m,$
we get
\begin{equation}
\label{equation:general-m<2beta}
m\le2\la(\beta-\ord_CD).
\end{equation}

Let
$g:\ \tilde{S}\rightarrow S$
be the blow-up of the point $p$, and let $G$ be the exceptional curve of $g$. We let $\tilde{C}$, $\tilde{F}_0$ and $\tilde{\Omega}$ be the proper transform of $C$, $F_0$ and $\Omega$ respectively on the surface $\tilde{S}$. We put
$$\tilde{p}=\tilde{C}\cap G,\ \tilde{m}=\text{mult}_{\tilde{P}}\tilde{\Omega}.$$
Note that $G$, $\tilde{C}$ and $\tilde{F}_0$ are three smooth curves intersecting pairwise transversally at $\tilde{p}$.

To show the pair \eqref{equation:many-blow-ups-log-pair} is lc
at $p$, it suffices to show the pair
$$(\tilde{S},(1-\beta+\la\ord_CD)\tilde{C}+\la\tilde{F}_0\ord_{F_0}D+\tilde{\Omega}+(\la\ord_{F_0}D+m-\beta+\la\ord_CD)G)$$
is lc at any point $q\in G$.

First, suppose that $q\neq \tilde{p}$. We then need to prove that the pair $$(\tilde{S},\tilde{\Omega}+(\la\ord_{F_0}D+m-\beta+\la\ord_CD)G)$$
is lc at $q$. Note that $(\la\ord_{F_0}D+m-\beta+\la\ord_CD)\le1$ by Lemma \ref{ordlemma} and \eqref{equation:general-m<2beta}, so we may apply Lemma \ref{ioalem} at $q$ and it suffices to prove
$$G.\tilde{\Omega}\leq1,$$
which is true since $G.\tilde{\Omega}=m\le1$ (recall \eqref{equation:general-m<2beta}).

To finish the proof, it then suffices to show that the pair
$$(\tilde{S},(1-\beta+\la\ord_CD)\tilde{C}+\la\tilde{F}_0\ord_{F_0}D+\tilde{\Omega}+(\la\ord_{F_0}D+m-\beta+\la\ord_CD)G)$$
is lc at $\tilde{p}$.

Let
$h:\ \hat{S}\rightarrow \tilde{S}$
be the blow up of $\tilde{p}$ and let $H$ be the exceptional curve of $h$. We let $\hat{C}$, $\hat{F}_0$, $\hat{G}$ and $\hat{\Omega}$ be the proper transform of $\tilde{C}$, $\tilde{F}_0$, $G$ and $\tilde{\Omega}$ respectively on the surface $\hat{S}$.
Then $\hat{C}$, $\hat{F}_0$ and $\hat{G}$ intersect transversally with $H$ at three different points.
Also notice that
$$2\la(\beta-\ord_CD)-m-\tilde{m}=\hat{F}_0.\hat{\Omega}\ge0,$$
so we get
\begin{equation}
\label{equation:m+tilde-m<2beta}
m+\tilde{m}\le2\la(\beta-\ord_CD).
\end{equation}
Using $\tilde{m}\le m$, we have
\begin{equation}
\label{equation:tilde-m<beta}
\tilde{m}\le\la(\beta-\ord_CD).
\end{equation}
And also, using Lemma \ref{ordlemma}, \eqref{equation:m+tilde-m<2beta} and \eqref{r7} we have (for small $\beta$)
\begin{equation}
\label{equation:tangential-coefficient<1}
(2\la\ord_{F_0}D+m+\tilde{m}-2\beta+2\la\ord_CD)\le1.
\end{equation}

Now to finish the proof, it is enough to show that
the log pair $(\hat{S},(1-\beta+\la\ord_CD)\hat{C}+\la\hat{F}_0\ord_{F_0}D+\hat{\Omega}+(\la\ord_{F_0}D+m-\beta+\la\ord_CD)\hat{G}+(2\la\ord_{F_0}D+m+\tilde{m}-2\beta+2\la\ord_CD)H)$
is lc at any point $o\in H$.

First, suppose that $o\notin\hat{C}\cup\hat{F}_0\cup\hat{G}$. Then we need to show $$(\hat{S},\hat{\Omega}+(2\la\ord_{F_0}D+m+\tilde{m}-2\beta+2\la\ord_CD)H))$$
is lc at $o$. By \eqref{equation:tangential-coefficient<1} and Lemma \ref{ioalem}, it is enough to show
$H.\hat{\Omega}\le1,$
but $H.\hat{\Omega}=\tilde{m}\le1$ (recall \eqref{equation:tilde-m<beta}).

Second, suppose that $o=H\cap \hat{F}_0$. Then we need to show
$$(\hat{S},\la\hat{F}_0\ord_{F_0}D+\hat{\Omega}+(2\la\ord_{F_0}D+m+\tilde{m}-2\beta+2\la\ord_CD)H)$$
is lc at $o$. By \ref{equation:tangential-coefficient<1} and Lemma \ref{ioalem}, it is enough to show
$$H.(\la\hat{F}_0\ord_{F_0}D+\hat{\Omega})\le1,$$
i.e.,
$\la\ord_{F_0}D+\tilde{m}\leq1,$
and this follows from Lemma \ref{ordlemma} and \eqref{equation:tilde-m<beta}.

Third, suppose that $o=H\cap \hat{G}$. Then we need to show
$$(\hat{S},\hat{\Omega}+(\la\ord_{F_0}D+m-\beta+\la\ord_CD)\hat{G}+(2\la\ord_{F_0}D+m+\tilde{m}-2\beta+2\la\ord_CD)H)$$
is lc at $o$. By \eqref{equation:tangential-coefficient<1} and Lemma \ref{ioalem}, it is enough to show
$$H.(\hat{\Omega}+\la\ord_{F_0}D+m-\beta+\la\ord_CD)\hat{G})\le1,$$
i.e.,
$\la\ord_{F_0}D+m+\tilde{m}-\beta+\la\ord_CD\le1,$
which holds by Lemma \ref{ordlemma} and \eqref{equation:m+tilde-m<2beta}.

Hence, to conclude the proof, it suffices to show the pair
$$(\hat{S},(1-\beta+\la\ord_CD)\hat{C}+\hat{\Omega}+(2\la\ord_{F_0}D+m+\tilde{m}-2\beta+2\la\ord_CD)H)$$
is lc at $o=H\cap \hat{C}$.
Note that $\mult_o\hat{\Omega}\le\tilde{m}\le1$, so by Corollary \ref{theorem:local-inequality-II}, it suffices to show
$$(\hat{C}.\hat{\Omega})_o\le\frac{(H.\hat{\Omega})_o}{((H.\hat{\Omega})_o-(\beta-\la\ord_CD))_+}(1-(2\la\ord_{F_0}D+m+\tilde{m}-2\beta+2\la\ord_CD))-(\beta-\la\ord_CD).$$

Now using Lemma \ref{ordlemma}, \eqref{r7}, \eqref{lambdaeq} and \eqref{equation:m+tilde-m<2beta}, it is clear that, for $\beta$ sufficiently small,
$$(2\la\ord_{F_0}D+m+\tilde{m}-2\beta+2\la\ord_CD)\le 1-\frac{\beta}{10}.$$
Also,
$(\hat{C}.\hat{\Omega})_o\le C.\Omega\le3,$
and using \eqref{equation:tilde-m<beta}, 
$(H.\hat{\Omega})_o\le H.\hat{\Omega}=\tilde{m}\le\la(\beta-\ord_CD).$
So by \eqref{xxmbeq} it suffices to show
$$3+(\beta-\la\ord_CD)\le\frac{\la(\beta-\ord_CD)}{\la(\beta-\ord_CD)-(\beta-\la\ord_CD)}\cdot\frac{\beta}{10},$$
i.e.,
$3+(\beta-\la\ord_CD)\le\frac{\la(\beta-\ord_CD)}{(\la-1)\beta}\cdot\frac{\beta}{10}.$
Using Lemma \ref{ordlemma} and \eqref{lambdaeq}, it is enough show
$3+\beta\le\frac{\beta-\frac{2}{3}\beta}{\beta^2/100}\cdot\frac{\beta}{10},$
i.e.,
$3+\beta\le\frac{10}{3},$
concluding the proof.
\end{proof}

\begin{claim}
\lb{claim6}
Suppose that $\overline{p_0}$ is blown up, then the pair \eqref{equation:many-blow-ups-log-pair} is lc at $p=C\cap F_0$.
\end{claim}

\begin{proof}
In this case, both $E_0$ and $F_0$ meet $C$ transversally at $p$.
Write
$$\la D=\la E_0\ord_{E_0}D+\la F_0\ord_{F_0}D+\la C\ord_CD+\Omega.$$
Put $m=\mult_p\Omega$. We have
$$
\begin{cases}
\la\beta+\la\ord_{E_0}D-\la\ord_{F_0}D-\la\ord_CD=\Omega.E_0\ge m,\\
\la\beta-\la\ord_{E_0}D+\la\ord_{F_0}D-\la\ord_CD=\Omega.F_0\ge m.\\
\end{cases}
$$
Summing up, we get
\begin{equation}
\label{equation:special-m<beta}
m\le\la(\beta-\ord_CD).
\end{equation}
Then by Lemma \ref{ordlemma}, $r\ge7$ and \eqref{lambdaeq}, we clearly have (for $\beta$ small enough)
\begin{equation}
\label{equation:special-coefficient<1}
(\la\ord_{E_0}D+\la\ord_{F_0}D+m-\beta+\la\ord_CD)\le1.
\end{equation}

Let
$g:\ \tilde{S}\rightarrow S$
be the blow up of $p$ and let $G$ be the exceptional curve of $h$. We let $\tilde{C}$, $\tilde{E}_0$, $\tilde{F}_0$ and $\tilde{\Omega}$ be the proper transform of $C$, $E_0$, $F_0$ and $\Omega$ respectively on the surface $\hat{S}$.
Then $\tilde{C}$, $\tilde{E}_0$ and $\tilde{F}_0$ intersect transversally with $H$ at three different points.

To show the pair \eqref{equation:many-blow-ups-log-pair} is lc at $p$, it is enough to show the pair
$$(\tilde{S},(1-\beta+\la\ord_CD)\tilde{C}+\la\tilde{E_0}\ord_{E_0}D+\la\tilde{F}_0\ord_{F_0}D+\tilde{\Omega}+(\la\ord_{E_0}D+\la\ord_{F_0}D+m-\beta+\la\ord_CD)G)$$
is lc at any point $q\in G$.

First suppose that $q\notin\tilde{C}\cup\tilde{E}_0\cup\tilde{F}_0$, then we need to show the pair
$$(\tilde{S},\tilde{\Omega}+(\la\ord_{E_0}D+\la\ord_{F_0}D+m-\beta+\la\ord_CD)G)$$
is lc at $q$. Using \eqref{equation:special-coefficient<1} and Lemma \ref{ioalem}, it suffices to show
$G.\tilde{\Omega}\leq1,$
but $G.\tilde{\Omega}=m\le1$ (recall \eqref{equation:special-m<beta}).

Second, suppose that $q=G\cap \tilde{E}_0$. Then we need to show the pair
$$(\tilde{S},\la\tilde{E_0}\ord_{E_0}D+\tilde{\Omega}+(\la\ord_{E_0}D+\la\ord_{F_0}D+m-\beta+\la\ord_CD)G)$$
is lc at $q$. Using \eqref{equation:special-coefficient<1} and Lemma \ref{ioalem}, it suffices to show
$$G.(\la\tilde{E_0}\ord_{E_0}D+\tilde{\Omega})\leq1,$$
i.e.,
$\la\ord_{E_0}D+m\le1,$
which is true by Lemma \ref{ordlemma}, \eqref{lambdaeq} and \eqref{equation:special-m<beta}. The proof for $q=G\cap \tilde{F}_0$ is similar.

So to finish the proof, it suffices to show the pair
$$(\tilde{S},(1-\beta+\la\ord_CD)\tilde{C}+\tilde{\Omega}+(\la\ord_{E_0}D+\la\ord_{F_0}D+m-\beta+\la\ord_CD)G)$$
is lc at $q=H\cap\tilde{C}$. Since $\mult_q\tilde{\Omega}\le m\le1$ by \eqref{equation:special-m<beta}, then by Corollary \ref{theorem:local-inequality-II}, it suffices to show
$$(\tilde{C}.\tilde{\Omega})_q\le\frac{(G.\tilde{\Omega})_q}{((G.\tilde{\Omega})_q-(\beta-\la\ord_CD))_+}(1-(\la\ord_{E_0}D+\la\ord_{F_0}D+m-\beta+\la\ord_CD))-(\beta-\la\ord_CD).$$

Now using Lemma \ref{ordlemma}, \eqref{r7}, \eqref{lambdaeq} and \eqref{equation:special-m<beta}, we have (for $\beta$ small enough)
$$(\la\ord_{E_0}D+\la\ord_{F_0}D+m-\beta+\la\ord_CD)\le1-\frac{\beta}{10}.$$
Also we have
$(\tilde{C}.\tilde{\Omega})_q\leq C.\Omega\le3,$
and by \eqref{equation:special-m<beta},
$(G.\tilde{\Omega})_q\le G.\tilde{\Omega}=m\le\la(\beta-\ord_CD).$
So by \eqref{xxmbeq} it suffices to show
$$3+(\beta-\la\ord_CD)\le\frac{\la(\beta-\ord_CD)}{\la(\beta-\ord_CD)-(\beta-\la\ord_CD)}\cdot\frac{\beta}{10},$$
i.e.,
$3+(\beta-\la\ord_CD)\le\frac{\la(\beta-\ord_CD)}{(\la-1)\beta}\cdot\frac{\beta}{10}$
Using Lemma \ref{ordlemma} and \eqref{lambdaeq}, it is enough show
$3+\beta\le\frac{\beta-\frac{2}{3}\beta}{\beta^2/100}\cdot\frac{\beta}{10},$
i.e.,
$3+\beta\le\frac{10}{3},$
concluding the proof.
\end{proof}

Claims \ref{claim1}--\ref{claim6} and Definition \ref{definition:delta}
imply that there exists $b:=b(r)$ such that for all rational $\be\in(0,b)$
we have
$\blct_k(S,(1-\be)C,-K_S-(1-\be)C)\ge 1+\frac{\be}{100}$ for all sufficiently large $k$.
Thus, $\blct_\infty(S,(1-\be)C,-K_S-(1-\be)C)\ge 1+\frac{\be}{100}$. Hence, Theorem \ref{main} follows from Theorem \ref{FujThm}.

\section*{Appendix: 
blct and the greatest Ricci lower bound}

\def\dzb#1{d\overline{z^{#1}}}  \def\bdz#1{d\overline{z^{#1}}}
\def\dbz{d\overline z}

\def\dzidzjb{dz^i\w\dbz^j}
\def\gij{g_{i\overline j}}
\def\opcit{\underbar{\phantom{aaaaa}}}

Let  $X$ be a Fano manifold.
A well-known result of Demailly states $\glct(X,-K_X)=\alpha(X)$, i.e., the global log canonical threshold coincides with Tian's $\alpha$-invariant \cite{CDS}. Here we show that the basis log canonical threshold of Fujita--Odaka coincides with Tian's $\beta$-invariant.
Recall the definition of the latter
$$
\beta(X):=\sup\{\; b \,:\, \Ric\,\o\ge b\o, [\o]=c_1(X)\},
$$
where $\Ric\,\omega:=-\i/2\pi\cdot\ddbar\log\det(\gij)$ denotes the
Ricci form of $\o=\i/2\pi\cdot\gij(z)\dzidzjb$.
This invariant was the topic of Tian's article \cite{TianGreatesLowerRicci}
although it was not explicitly defined there, but 
was first explicitly defined by one of us in \cite[(32)]{R08}, \cite[Problem 3.1]{R07}
and was later further studied by Sz\'ekelyhidi
\cite{GaborGreatesLowerRicci}, Li \cite{Litoric},
Song--Wang \cite{SongWang}, and Cable \cite{Cable}.

\begin{theorem}
\label{theorem:delta-equal-R}
On a Fano manifold  $X$, $\be(X)=\min\{\blct_\infty(X,-K_X),1\}.$
\end{theorem}

Some special cases of this are known.
First, the result is inspired by
the work of Blum-Jonsson 
who derived this identity 
in the special toric case
by directly computing $\blct_\infty(X,-K_X)$ \cite[Corollary 7.19]{BJ17} and 
observing it coincides with Li's formula for $\be(X)$ for toric $X$ \cite{Litoric}.
Second, Theorem \ref{theorem:delta-equal-R} is known
if $(X,-K_X)$ is semistable in an algebraic/analytic sense.
Indeed, Li \cite{ChiLi-Criterion} showed that $\be(X)=1$
if and only if the Mabuchi energy is bounded below solving
a problem posed by one of us  \cite[Problem 3.1]{R07}.
He also showed, using \cite{CDS,Tian2015}, that this happens
if and only if $(X,-K_X)$ is K-semistable. On the
other hand, by Fujita--Odaka and Blum-Jonsson \cite{FujitaOdaka,BJ17}
$(X,-K_X)$ is K-semistable if and only if
$\blct_\infty(X,-K_X)\ge1$.
Below we give a short proof of Theorem \ref{theorem:delta-equal-R} 
in the remaining case, i.e., when $(X,-K_X)$ is K-unstable,
namely, when $\be(X)\in(0,1)$ 
(
$\be(X)$ is positive by
the Calabi--Yau theorem),
so Theorem \ref{theorem:delta-equal-R}
reduces to the formula 
\beq
\lb{thmeq}
{\blct_\infty(X,-K_X)}=\be(X).
\eeq
Our strategy will be to use the scaling property
$
b^{-1}{\blct_\infty(X,-K_X)}=
\blct_\infty(X,-bK_X)
$
for $0<b\in\QQ$
\cite[Remark 4.5]{BJ17} and show 
$
\blct_\infty(X,-bK_X)\ge 1 
$
for $b\in(0,\be(X))\cap\QQ$
and $
\blct_\infty(X,-bK_X)\le 1 
$
for 
$b\in(\be(X),1)\cap\QQ$.

The proof makes use of KEE metrics (see \S\ref{KEESec} and \cite{R14}
for background).
In the edge setting, one has an analogue of $\beta(X)$ due
to Donaldson \cite{Don2009} and Li--Sun \cite{ChiLi-Criterion}:
for all $m\in\NN$ sufficiently large, choose $\Delta_m\in|-m K_X|$  a smooth divisor
(exists by Bertini's theorem), denote
by $[\D_m]$ the current of integration along $\D_m$, and set
$$
\be(X,\Delta_m/m):=
\sup\{b>0\,:\, \Ric\,\omega=b\omega+(1-b)[\Delta_m]/m, 
[\omega]=c_1(X)
 \}.
$$
It is known that  \cite{SongWang}
(see also \cite[Corollary 2.4]{C.LiYTDcorresponce})
\beq
\lb{SWeq}
\lim_m \be(X,\Delta_m/m))=\be(X).
\eeq
In particular, fixing any $b\in(0,\be(X))\cap\QQ$, there is $m_0\in\NN$
such that $\be(X,\Delta_m/m))>b$ for all $m\ge m_0$, and by
definition and 
there are KEE metrics with Ricci curvature $b$ and
with angles $2\pi(1-b)/m$ along $\D_m$
(here we use
 \cite[Theorem 1.1]{LiSun} 
that guarantees the interval of such values of $b$ is connected)
and therefore $(X,(1-b)\D_m/m,-bK_X)$
are log K-semistable  \cite[Corollary 1.12]{LiSun}.
Thus by \cite[Corollary 4.8]{CP-delta-and-stability-for-log-case}, we have
$$
\blct_\infty(X,(1-b)\D_m/m,-bK_X)\ge1.
$$
By definition,
$\blct_\infty(X,-bK_X)\ge\blct_\infty(X,(1-b)\Delta_m/m,-bK_X).$
Thus, $\blct_\infty(X,-bK_X)\ge 1$ for each $b\in(0,\be(X))\cap\QQ$,
so $\blct_\infty(X,-K_X)\ge \be(X)$.

For the other direction of \eqref{thmeq} we make use
an algebraic counterpart of \eqref{SWeq}:

\begin{lemma}
\label{lemma:log-delta-converge-to-delta}
For $b\in(0,1)\cap\QQ$,
$\lim_{m}\blct_\infty(X,(1-b)\Delta_m/m,-bK_X)=\blct_\infty(X,-bK_X).$
\end{lemma}

\begin{proof}
As just noted, one direction follows from the definitions.
%

For the reverse direction, first fix $k$, and then let $c\in\big(0,\blct_k(X,-bK_X)\big)$. 
Let $D\sim_{\mathbb{Q}}-bK_X$ be a $k$-basis divisor, so 
$(X,cD)$ is lc. Observe that, of course, $(X,\D_m)$ is lc (as $\D_m$ is smooth).
Recall that if $(X,A)$ and $(X,B)$ are lc then so is $(X,(1-\delta)A+\delta B)$
for any $\delta\in(0,1)$  \cite[Remark 2.1]{CheltsovToolBox}.
Thus $(X,(1-\delta)cD+\delta\D_m)$ is lc. Put $\delta=(1-b)/m$
to obtain that 
$$
\blct_k(X,(1-b)\Delta_m/m,-bK_X)\ge 
(1-O(1/m))\blct_k(X,-bK_X).
$$
Now let first $k$ and then $m$ tend to infinity to conclude.
\end{proof}

Thus, suppose that 
$\blct_\infty(X,-bK_X)>1$ for some $b\in
(\be(X),1)\cap\QQ$.
By Lemma \ref{lemma:log-delta-converge-to-delta}
$\blct_\infty\big(X,(1-b)\Delta_m/m,-bK_X\big)>1$
and $(X,(1-b)\Delta_m/m,-bK_X)$ is uniformly log K-stable 
for all sufficiently large $m$
\cite[Corollary 4.8]{CP-delta-and-stability-for-log-case}.
So it follows from \cite{CDS,Tian2015} (see also \cite{TianWang})
that there exists a KEE metric associated to this triple, i.e., that
$\be(X,\D_m/m)\ge b$, contradicting \eqref{SWeq}.
Thus, $\blct_\infty(X,-bK_X)\le 1$, i.e.,
$\blct_\infty(X,-K_X)\le b$ for all $b\in(\be(X),1)\cap\QQ$.
This concludes the proof of \eqref{thmeq}
and hence of Theorem \ref{theorem:delta-equal-R}.

\begin{remark}
\label{remark:alternative-proof}
In the last paragraph one may also use 
\cite[Corollary 2.11]{BJ18} to obtain the polarized pair 
$\big(X,-bK_X\big)$ is K-semistable in the adjoint sense, hence twisted K-semistable in the sense of \cite{Dervan16} (see \cite[Proposition 8.2]{BHJ17}). So \cite[Proposition 10]{DG} guarantees that 
for some $b\in (\be(X),1)$, we can find two K\"ahler forms $\omega,\alpha$ cohomologous to $c_1(X)$ such that
$\Ric\,\omega=b\omega+(1-b)\alpha,$
that is again a contradiction.
\end{remark}

%


{\sc University of Edinburgh}

{\tt i.cheltsov@ed.ac.uk}

\smallskip

{\sc University of Maryland}

{\tt yanir@umd.edu}

{\sc Peking University and University of Maryland}

{\tt kwzhang@pku.edu.cn}

\end{document}